\documentclass[12pt]{iopart}
\usepackage[utf8x]{inputenc}
\usepackage{enumitem}
\usepackage{iopams}
\usepackage{setstack}
\usepackage{amsthm}
\usepackage{amsopn}
\usepackage{xcolor}
\usepackage{cite}
\usepackage{tikz}

\theoremstyle{plain}
\newtheorem{theorem}{Theorem}[section]
\newtheorem{lemma}[theorem]{Lemma}
\newtheorem{proposition}[theorem]{Proposition}
\newtheorem{corollary}[theorem]{Corollary}
\theoremstyle{remark}
\newtheorem{remark}[theorem]{Remark}
\newtheorem{example}[theorem]{Example}
\theoremstyle{definition}

\newtheorem{definition}[theorem]{Definition}

\newcommand{\abs}[1]{|{#1}|}
\newcommand{\norm}[2][]{\|{#2}\|_{#1}}
\newcommand{\inner}[3][]{\langle{#2},{#3}\rangle_{#1}}

\newcommand{\set}[2]{\{ {#1}\ :\ {#2} \}}

\newcommand{\wstarto}{\overset{\ast}{\rightharpoonup}}
\newcommand{\conv}[1]{\ensuremath{\overset{#1}{\rightarrow}}}

\newcommand{\placeholder}{\,\cdot\,}
\newcommand{\RR}{\mathbb{R}}
\newcommand{\NN}{\mathbb{N}}

\newcommand{\topin}{\ensuremath{\tau_{IN}}}
\newcommand{\ball}{\ensuremath{\mathcal{B}}}
\newcommand{\seqstruc}{{\mathcal{S}}}

\DeclareMathOperator*{\argmin}{argmin}
\DeclareMathOperator{\dom}{dom}

\DeclareMathOperator{\id}{id}
\DeclareMathOperator{\graph}{gr}
\DeclareMathOperator{\range}{range}
\DeclareMathOperator{\spann}{span}

\begin{document}
\title[Variational regularization]{Necessary conditions for variational regularization schemes}
\author{Dirk Lorenz$^1$ and Nadja Worliczek$^1$}
\address{$^1$ Institute for Analysis and Algebra, TU Braunschweig, 38092 Braunschweig, Germany}
\eads{\mailto{d.lorenz@tu-braunschweig.de}, \mailto{n.worliczek@tu-braunschweig.de}}

\begin{abstract}
  We study variational regularization methods in a general framework,
  more precisely those methods that use a discrepancy and a regularization
  functional. While several sets of sufficient conditions are known to
  obtain a regularization method, we start with an investigation of
  the converse question: How could necessary conditions for a
  variational method to provide a regularization method look like? To this end,
  we formalize the notion of a variational scheme and start with comparison
  of  three
  different instances of variational methods. Then we focus on the
  data space model and investigate the role and interplay of the
  topological structure, the convergence notion and the discrepancy
  functional. Especially, we deduce necessary conditions for the
  discrepancy functional to fulfill usual continuity assumptions.
  The results are applied to discrepancy functionals given by Bregman
  distances and especially to the Kullback-Leibler divergence. 
\end{abstract}

\ams{49N45, 54A10, 54A20}

%\textbf{Keywords:} Variational regularization, Tikhonov
%regularization, necessary conditions, general discrepancy term, Bregman distances

\section{Introduction}
By ``variational regularization'' we mean every method that is used to
approximate an ill-posed problem by well-posed minimization problems.
We start with a mapping $F:X\to Y$ between two sets $X$ and $Y$ and
equations
\[
F(x)=y.
\]
A common problem with inverse problems is that of instability,
i.e.~that arbitrary small disturbances in the right hand side $y$
(e.g.~by replacing a ``correct'' $y$ in the range of $F$ with one in
an arbitrarily small neighborhood) may lead to unwanted effects such as
that no solution exists anymore or that solutions with perturbed
right hand side differ arbitrarily from the true solutions. In
topological spaces $X$ and $Y$ we can formulate the problem of
instability more precisely: The equation
$F(x^{\text{exact}})=y^{\text{exact}}$ is unstable, if there exists a
neighborhood $\mathcal{U}$ of $x^{\text{exact}}$ such that for all
neighborhoods $\mathcal{V}$ of $y^{\text{exact}}$ there exists
$y^\delta\in\mathcal{V}$ such that
$F^{-1}(y^\delta)\cap\mathcal{U}=\emptyset$ (cf.~\cite{ivanov1969topological,kabanikhin2011inverseproblems}).

Variational regularization methods replace the equation
$F(x)=y$ by a minimization problem for an (extended) real valued
functional such that the minimizers are suitable approximate solutions
of the equation. The most widely used variational method is Tikhonov
regularization~\cite{tikhonov1963regularization}, but other methods
are used as well. Starting from a detailed analysis of this method in
Hilbert spaces, there are several recent studies on Tikhonov
regularization in the context of more general spaces like Banach
spaces~\cite{resmerita2005regbanspaces,resmerita2006nonquadreg,hofmann2007convtikban,scherzer2009variationalmethods} or
even topological spaces~\cite{poeschl2008diss,flemming2010nonmetric,flemming2011diss,grasmair2012multiparameter,werner2012tikhonovpoisson}.
%Here: Concrete Applications
Especially, the discrepancy functional, which measures the distance
between the measured data and the reconstructed data, have come into
the focus of recent research: A Poisson noise model motivates the use
of a Kullback-Leibler divergence and is applied in fluorescence
microscopy and optical/infrared
astronomy~\cite{bertero2009poissondeblurring}, for inverse
scattering problems and phase retrieval
problems~\cite{hohage2013irgnpoisson} and for STED- and 4Pi-microscopy
and positron emission tomography~\cite{brune2013fbemtv}. Moreover,
a kind of Burg entropy is used for multiplicative noise which has
applications to remove speckle noise in synthetic aperture radar
imaging~\cite{aubert2008multiplicativenoise}.
By now, a quite general set of sufficient assumptions is available under which
Tikhonov regularization has the desired regularizing properties,
i.e.~stable solvability of the minimization problems and suitable
approximation of the true solution if the noise vanishes.
These sufficient assumptions are helpful to
check if a chosen setting for variational regularization is indeed
suited. On the other hand, when designing a regularization method it would be
helpful to know in advance which setting works and which is not going
to work. Hence, in this paper we begin with a study of the converse
analysis and aim at providing \emph{necessary conditions} on
variational methods such that regularization is achieved. Such conditions
would also be helpful in designing new variational methods as they
rule out several options. Moreover, necessary conditions are a
further step towards the understanding of the nature of variational
regularization.

We remark that we are aware that necessary conditions can not be
expected to be very strong (as an example, a minimization problem can
be changed quite arbitrary without changing
the minimizer itself). However, there are already a few results of
this flavor known in specialized contexts which we list here:
\begin{theorem}[No uniform bounded linear regularization, \protect{\cite[Remark 3.5]{engl1996inverseproblems}}]
  \label{thm:no-bounded-reg}
  If the linear and bounded operator $F:X\to Y$ between Hilbert spaces
  $X$ and $Y$ does
  not have closed range and $(L_\alpha)_{\alpha>0}$ is a family of
  linear and bounded operators from $Y$ to $X$ such that for all $x\in
  X$ it holds that $L_\alpha Fx$ converges to $x$ for $\alpha\to 0$,
  then $(\norm{L_\alpha})$ is unbounded.
\end{theorem}
In other words, linear regularization methods are necessarily not
uniformly bounded.

The next example of a necessary condition deals with the problem of
parameter choice. We need the Moore-Penrose pseudo-inverse $F^\dag$ of
a bounded linear mapping between Hilbert spaces,
cf.~\cite{benisrael2003geninverse}.
\begin{theorem}[Bakushinskii Veto, \protect{\cite{bakushinksii1984nonregularization}}]
  \label{thm:bakushinskii}
  Let $F:X\to Y$ be a bounded linear operator between Hilbert spaces
  and $(L_\alpha)_{\alpha>0}$ be a family of continuous mappings from
  $Y$ to $X$. If there is a mapping $\alpha:Y\to{]}0,\infty{[}$ such
  that
  \[
  \limsup_{\delta\to 0}\set{\norm{L_{\alpha(y^\delta)}y^\delta -
      F^\dag y}}{y^\delta\in Y,\ \norm{y-y^\delta}\leq\delta}=0
  \]
  then $F^\dag$ is bounded.
\end{theorem}
In other words, parameter choice rules which are regularizing (in the sense of~\cite[Def.~3.1]{engl1996inverseproblems}) for ill-posed problems (i.e. unbounded $F^\dag$)
necessarily need to use the noise level.

An example for a-priori parameter choice rules was proven by Engl:
\begin{theorem}[Decay conditions for a-priori parameter choice rules
  for linear methods,~\protect{\cite[Prop.~3.7]{engl1996inverseproblems}
  and\cite{engl1981necessary})}] Let $F$ and $(L_\alpha)$ be as in
  Theorem~\ref{thm:no-bounded-reg}, and
  $\alpha:{]}0,\infty{[}\to{]}0,\infty{[} $ be an a-priori parameter
  choice rule. Then it holds that
  \[
  \limsup_{\delta\to 0}\set{\|L_{\alpha(\delta)} y^\delta - F^\dag
    y\|}{y^\delta\in Y,\ \norm{y-y^\delta}\leq\delta}=0
  \]
  if and only if
  \[
  \lim_{\delta\to 0}\alpha(\delta) =
  0\quad\text{and}\quad\lim_{\delta\to
    0}\delta\norm{L_{\alpha(\delta)}} = 0.
  \]
\end{theorem}
In other words, a-priori parameter choice rules necessarily need to fulfill certain decay conditions.

Finally we mention the ``converse results'' from~\cite{neubauer1997saturationtikhonov}
which say that for
Tikhonov regularization in Hilbert spaces certain convergence rates
imply that certain source conditions are
fulfilled
(see~\cite{flemming2011converse} for generalization to other
regularization methods).
% In other words, a certain convergence rate
% necessarily implies a certain source condition.

 Before we start our investigation of necessary conditions for
 variational regularization in Section~\ref{sec:nec-cond}, we start
 with a section in which we formalize the notation of a ``variational
 scheme'' and investigate a few different variational methods.

\section{Variational schemes: Tikhonov, Morozov, and Ivanov}
\label{sec:var-reg}
In this section we formalize the notion of a variational scheme which
can be used to build variational regularization methods.  We start by
fixing the ingredients of an \emph{inverse problem}: In this paper we
take the point of view, that an inverse problem consists of a mapping
$F:(X,\tau_X) \to (Y,\tau_Y)$ between two topological spaces, usually
called \emph{forward operator}. The space $X$ is the \emph{solution
  space} and $Y$ is called \emph{data space}. We further assume that
$F$ is continuous, i.e.~the forward problem (calculating the
data for some given solution) is well-posed. In contrast, the solution of
an equation $F(x)=y$ for some given $y$ does not need to be
well-posed.
As in~\cite{ivanov1969topological} and the more recent references
\cite{grasmair2011residual,grasmair2012multiparameter} we use
topological spaces since the functionals we consider do no take any
linear structure into account which would justify the use of linear or
normed spaces. 

A variational scheme consists of all ingredients which are needed to
classify and analyze the associated minimization problems and their
minimizers under perturbations of the data $y$. Hence, it should
encode information about the notions of convergence, ``proximity'', and
the objective functional to be minimized. However, we do not allow for
totally arbitrary objective functionals but we rather use the
intuition that a variational scheme involves two functionals: a
``similarity measure'' or ``discrepancy functional'' $\rho$ and a
``regularization functional'' $R$. The functional $\rho$ is used to
measure ``similarity'' in the data space in the sense that
$\rho(F(x),y)$ is small if $x$ explains the data $y$ well. The
functional $R$ on the solution space is used to measure how well $x$
fits prior knowledge in the sense that $R(x)$ is small for an $x$
which fulfills the prior knowledge well.

\begin{definition}[Variational scheme]
  \label{def:variational_scheme}
  By a \emph{variational scheme} for a given inverse problem
  $F:(X,\tau_X)\to(Y,\tau_Y)$ we understand a tuple
  $\mathcal{M}=\left(\rho, R,\seqstruc\right)$, consisting of
  \begin{itemize}
  \item the \emph{discrepancy functional} $\rho:Y\times Y\rightarrow
    [0,\infty]$, for which we assume that $\rho(y,y)=0$ for all $y\in
    Y$,
  \item the \emph{regularization functional} $R:X\rightarrow
    [0,\infty]$, and
  \item a sequential convergence structure $\seqstruc$ on $Y$.
    
    That is, $\seqstruc$
    is a mapping which maps any element in $Y$ to a set of
    sequences in $Y$ such that the constant sequence $(y)$ is an
    element of $\seqstruc(y)$ and that if a sequence is in $\seqstruc(y)$
    then so does any of its subsequences. Usually, we denote
    $(y_n)_n\in\seqstruc(y)$ by $y_n\conv{\seqstruc} y$ and say that $(y_n)$
    converges to $y$ (with respect to $\seqstruc$), see also~\cite[\S
1.7]{beattie2002convergence}.
  \end{itemize}
\end{definition}

 While  most ingredients of a
variational scheme are standard, we remark on the sequential convergence
structure $\seqstruc$: Often decaying noise is described in terms of
norm-convergence, a notion which is not available here and sometimes
may not even be
appropriate (see, e.g.~\cite{flemming2011diss}
and~\cite{poeschl2008diss}).  Therefore, the sequential convergence
structure will be used to describe ``vanishing noise'' in $Y$,
i.e.~the vanishing of noise is modeled by convergence of a sequence
$(y_n)$ to noise free data $y$ w.r.t.~$\seqstruc$.  Of course, a
topology induces a sequential convergence structure but not all
sequential convergence structures are topological (e.g. pointwise
almost everywhere convergence is not induced by a
topology~\cite{ordman1966nottopological}). Moreover, a sequential
convergence structure induced by a topology may not encode all
information of the topology (consider the case of the sequence space
$\ell^1$ for which, by Schur's Theorem~\cite{diestel1984banachspaces},
the weak topology induces the same convergence structure as the norm
topology although the former is strictly weaker than the latter). Note
that we do not assume that convergence w.r.t.~$\seqstruc$ is
topological since this is not used in standard proofs for regularizing
properties (e.g.~\cite{hofmann2007convtikban}).  Moreover, the
topology $\tau_Y$ may induce a different convergence structure which
is more tied to the mapping properties of $F$. 
% In the case of a
% Banach space $Y$ one can think of the following situation which is for
% example used in~\cite{hofmann2007convtikban}: The sequential
% convergence structure is given by the convergence with respect to the
% norm in $Y$ and $\tau_Y$ is the weak topology on $Y$.
Of course, there
will be further relations between $\tau_Y$, $\seqstruc$ and $\rho$ in
the following, and indeed, Section~\ref{sec:nec-cond} mainly deals
with these relations, but for the general variational scheme we keep
them mostly unrelated.

We mention that we included the value $\infty$ in the range of the
discrepancy functional $\rho$ and the regularization functional $R$ to
model that certain data may be considered ``incomparable'' or that
certain solutions may be deemed to be impossible. As usual, the value $\infty$ is
excluded for minimizers by definition and we use the notation $\dom R
= \set{x}{R(x)<\infty}$ (similarly for $\rho$).

Variational regularization methods can be build from variational
schemes as follows.  Instead of solving $F(x)=y$ we aim at two goals:
Find an $x\in X$ such that $x$ explains the data $y$ well, in the
sense that $\rho(F(x),y)$ is small, and $x$ fits to our prior knowledge
in the sense that $R(x)$ is small.  In other words: We have two
objective functionals $x\mapsto \rho(F(x),y)$ and $x\mapsto R(x)$ which
we would like to ``jointly minimize'' and such problems go under the
name of ``vector optimization''. A core notion
there is that of ``Pareto-optimal solutions'', i.e., solutions $x^*$
such that there does not exist an $x$ such that $R(x)\leq R(x^*)$ and
$\rho(F(x),y)\leq \rho(F(x^*),y)$ and one of both inequalities is
strict~\cite[\S 4.7]{boyd2004convexoptimization}.  Note that for
``exact data'', i.e. $y^{\text{exact}}$ in the range of $F$, the
notion of Pareto optimality induces a notion of generalized solutions
of the equation $F(x)=y$ (see~\cite{flemming2011diss} for a slightly
different notion):
\begin{definition}
  Let $\left(\rho,R,\seqstruc\right)$ be a variational scheme for
  $F:(X,\tau_X)\to (Y,\tau_Y)$ and $y^{\text{exact}}$ be in the range
  of $Y$. We say that $\bar x$ is a \emph{$\rho$-generalized
    $R$-minimal solution} of $F(x)=y^{\text{exact}}$ if $\rho(F(\bar
  x),y^{\text{exact}})=0$ and $R(\bar x) =
  \min\set{R(x)}{\rho(F(x),y^{\text{exact}})=0}$.
\end{definition}

Using the two objective functionals $\rho(F(\cdot),y)$ and $R$ we can
build at least three different minimization problems which aim at
finding Pareto optimal solutions. These three problems are well known
in the inverse problems community and in fact can be traced back to
the pioneering works in the Russian school: \emph{Tikhonov
regularization}~\cite{tikhonov1963regularization} sets $T_{\alpha,
  y}(x):=\rho(F(x),y)+\alpha R(x)$ for some $\alpha>0$ and considers
\begin{equation}
  \label{eq:tikhonov}
  % \tag{T$_\alpha$}
  T_{\alpha, y}(x)\rightarrow \min_{x\in X}.
\end{equation}
In other words: Choose a weighting between ``good data fit'' and
``good fit to prior knowledge'' and minimize the weighted objective
functional. \emph{Ivanov regularization}~\cite{ivanov1962wellposed} uses $\tau >
0$ and considers
\begin{equation}
  \label{eq:ivanov}
  % \tag{I$_\tau$}
  \rho(F(x),y)\rightarrow \min_{x\in X}\quad \text{s.t.}\quad R(x)\leq\tau.
\end{equation}
In other words: Choose the solution with the best data-fit which also
fits the prior knowledge up to a predefined amount.  Finally, \emph{Morozov
regularization}~\cite{morozov1967discrepancy} uses $\delta >0$ and
considers
\begin{equation}
  \label{eq:morozov}
  % \tag{M$_\delta$}
  R(x)\rightarrow \min_{x\in X}\quad \text{s.t.} \quad\rho(F(x),y)\leq \delta.
\end{equation}
In other words: Choose the solution which fits best the prior
knowledge among the ones which explain the data up to a
predefined amount.

These methods are treated and compared
e.g.~in~\cite[Ch.~3.5]{ivanov2002linearillposed}
in the case of Banach spaces and $\rho(F(x),y) = \norm{F(x)-y}^p$ and
$R(x) = \norm{Lx}^q$ with a (possibly unbounded) linear operator $L$ (where~(\ref{eq:ivanov}) is called ``method of quasi-solutions''
and~(\ref{eq:morozov}) goes under the name ``method of the residual'').
We state a result on the relation of the minimizers of these methods
in our abstract framework of a variational scheme without any
convexity assumptions on $R$ or $\rho$.
\begin{theorem}
  \label{thm:relations_var_methods}
  Let $(\rho,R,\seqstruc)$ be a variational scheme for
  $F:(X,\tau_X)\to (Y,\tau_Y)$.
  \begin{enumerate}
  \item If there exists a unique solution $x_\tau$
    of~(\ref{eq:ivanov}), then it is the unique solution
    of~(\ref{eq:morozov}) with $\delta= \rho(F(x_\tau),y)$.
  \item If there exists a unique solution $x_\delta$
    of~(\ref{eq:morozov}), then it is the unique solution
    of~(\ref{eq:ivanov}) with $\tau=R(x_\delta)$.
  \item If there exists a unique solution $x_\alpha$ of~(\ref{eq:tikhonov})
    then it solves~(\ref{eq:ivanov}) with $\tau=R(x_\alpha)$ and
    (\ref{eq:morozov}) with $\delta=\rho(F(x_\alpha),y)$.
  \end{enumerate}
\end{theorem}
\begin{proof}
  (i) With $\delta=\rho(F(x_\tau),y)$, it is clear that $x_\tau$ is
  feasible for the optimization problem~(\ref{eq:morozov}) and the
  objective value is $R(x_\tau)$. Assume that there is a solution
  $\bar x\neq x_\tau$ of~(\ref{eq:morozov}) with $R(\bar x) \leq
  R(x_\tau)\leq\tau$. Then, $\bar x$ would be feasible
  for~(\ref{eq:ivanov}) with objective $\rho(F(\bar x),y)\leq\delta =
  \rho(F(x_\tau),y)$ which is a contradiction to the uniqueness of the
  solution $x_\tau$.  The proof of (ii) mimics the proof of (i).  For
  (iii) again, assume that there exists a solution $\bar x\neq
  x_\alpha$ of~(\ref{eq:ivanov}). Then, one sees that
  $T_{\alpha,y}(\bar x)\leq T_{\alpha,y}(x_\alpha)$ contradicting the
  uniqueness of $x_\alpha$. The proof is similar for the last claim.
\end{proof}
We remark that the missing implications in
Theorem~\ref{thm:relations_var_methods} are not true without
additional assumptions.
\begin{example}[Unique Ivanov and Morozov minimizers need not to be Tikhonov minimizers]
  \label{ex:iv_notimplies_tikh}
  We illustrate this by a simple
  one-dimensional example: Let $X=Y=\RR$, $F=\id$ and consider the
  regularization functional $R(x) = \abs{x+1}$ (saying that the
  solution should be close to $-1$) and as discrepancy functional the
  so-called Bregman distance with respect to the strictly convex
  function $x\mapsto x^4$, i.e. $\rho(x,y) = y^4-x^4 - 4x^3(y-x)$. We
  choose $\tau=1$ and $y=1$ and obtain $x_\tau=0$ as the unique solution
  of~(\ref{eq:ivanov}) (which is also the unique
  solution of~(\ref{eq:morozov}) with $\delta=1$). But there is no
  $\alpha>0$ such that $x=0$ is a minimizer of $T_{\alpha,1}(x) = \rho(x,1)
  + \alpha\abs{x+1}$
  (cf.~Figure~\ref{fig:example_iv_notimplies_tikh}).
    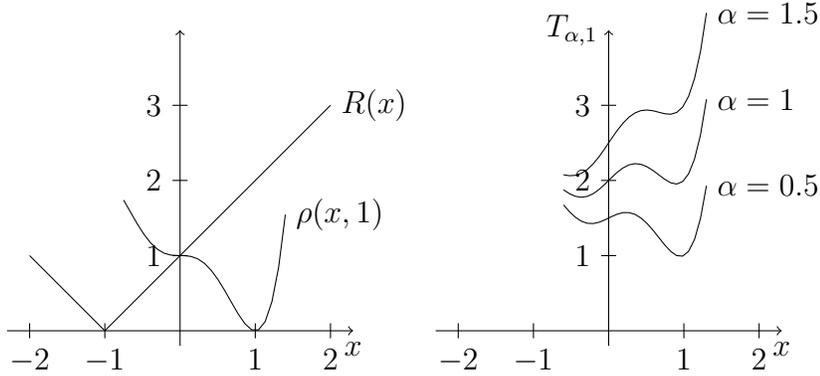
\begin{figure}[tb]
    \centering
    \begin{tikzpicture}
      \draw[->] (-2.3,0) -- (2.3,0) node[below]{$x$};
      \draw[->] (0,-0.2) -- (0,4);
      
      \foreach \x in {-2,-1,1,2}
      \draw(\x,0.1) -- (\x,-0.1) node[below]{$\x$};
      \foreach \y in {1,2,3}
      \draw(0.1,\y) -- (-0.1,\y) node[left]{$\y$};
      
      \draw[domain=-0.75:1.4] plot (\x, {3*\x^4 - 4*\x^3 +1}) node[right]{$\rho(x,1)$}; 
      \draw[domain=-2:2] plot (-2,1) -- (-1,0) -- (2,3) node[right]{$R(x)$};
      
    \end{tikzpicture}
    \begin{tikzpicture}
      \draw[->] (-2.3,0) -- (2.3,0) node[below]{$x$};
      \draw[->] (0,-0.2) -- (0,4) node[left]{$T_{\alpha,1}$};
      
      \foreach \x in {-2,-1,1,2}
      \draw(\x,0.1) -- (\x,-0.1) node[below]{$\x$};
      \foreach \y in {1,2,3}
      \draw(0.1,\y) -- (-0.1,\y) node[left]{$\y$};

      \draw[domain=-.6:1.3] plot (\x, {3*\x^4 - 4*\x^3 +1  + 0.5*abs(\x+1)}) node[right]{$\alpha = 0.5$};
      \draw[domain=-.6:1.3] plot (\x, {3*\x^4 - 4*\x^3 +1  + 1*abs(\x+1)}) node[right]{$\alpha = 1$};
      \draw[domain=-.6:1.3] plot (\x, {3*\x^4 - 4*\x^3 +1  + 1.5*abs(\x+1)}) node[right]{$\alpha = 1.5$};
    \end{tikzpicture}
    \caption{Illustration of Example~\ref{ex:iv_notimplies_tikh}}
    \label{fig:example_iv_notimplies_tikh}
  \end{figure}
\end{example}

In the above examples it holds that $x_\tau$ is a stationary point of
the mapping $x\mapsto \rho(F(x),y)$. Note that the precise form of $R$
is not important in this example,
several other $R$ with $R'(0)>0$ would also work.
Indeed, we can deduce from the next proposition that it is necessary
for $x_\tau$ to be also a (local) Tikhonov
minimizer that not both of these properties are fulfilled.
\begin{proposition}[\protect{\cite[Thm.~4.13]{poeschl2008diss}}]
  \label{prop:local-tikhonov}
  Let $(\rho,R,\seqstruc)$ be a variational scheme for\\
  $F:(X,\tau_X)\to (Y,\tau_Y)$, $X$ be a normed space and
  let $y\in Y$.  Furthermore, assume that the mappings $f(x) =
  \rho(F(x),y)$ and $R$ obey directional derivatives $R'(x^*;v)$ and
  $f'(x^*;v)$ for all directions $v\in X$.
  
  If $x^*$ is a local minimizer of $T_{\alpha,y}$ for some 
  $\alpha >0$ then for every $v$ it holds that
  \[
  -\alpha R'(x^*;v) \leq f'(x^*;v).
  \]
  Moreover, if the directional derivatives of $f$ and $R$ at $x^*$ are
  linear in $v$ and $R'(x^*;\cdot)\neq 0$, then $f'(x^*;\cdot)\neq 0$.
\end{proposition}
% \begin{proof}
%   Since $x^*$ is a local minimizer of $T_{\alpha,y}$, it holds, for
%   $\epsilon>0$ small enough,
%   \[
%   \alpha(R(x^*) - R(x^* + \epsilon v)) \leq f(x^* +
%   \epsilon v) - f(x^*).
%   \]
%   Dividing by $\epsilon$ and passing to the limit $\epsilon \to 0$
%   proves the first assertion. If $R'(x^*) \neq 0$ then there is $v$
%   such that $R'(x^*)v<0$ and it follows that $f'(x^*)v>0$; hence,
%   $f'(x^*)\neq 0$.
% \end{proof}

% Note that for $v$ such that $-R'(x^*)v<0$ it follows that $R(x^* +
% \epsilon v) < R(x^*)$ for small $\epsilon$. Hence, if $x^*$ would be a
% solution of~(\ref{eq:ivanov}) (with $\tau = R(x^*)$),
% $x^*+\epsilon v$ would be feasible
% for~(\ref{eq:ivanov}) and hence, $f(x^*+\epsilon v)\geq f(x^*)$ for
% small $\epsilon$.

In other words: If we have a solution $x^*$ of~(\ref{eq:ivanov}) with
$R'(x^*;\cdot)\neq 0$ which is also a
local minimizer of $T_{\alpha,y}$ then it is not stationary for
$x\mapsto \rho(F(x),y)$.

\begin{remark}
  Under convexity assumptions on $f(x) = \rho(F(x),y)$ and $R$ one can
  show that Ivanov minimizers (or Morozov minimizers) are indeed also
  Tikhonov minimizers for some parameter $\alpha>0$ if they are not
  minimizers of the constraint. This is related to the fact that the
  subgradients of convex functions describe the normal vectors to the
  sublevel sets of the respective function, see
  e.g.~\cite{schirotzek2007nonsmoothanalysis}.
\end{remark}

Although the variational
problems~(\ref{eq:tikhonov}),~(\ref{eq:ivanov}),
and~(\ref{eq:morozov}) share their solutions under the circumstances
presented above, they often differ with respect to their practical
application.

It has been remarked already in early works (see,
e.g.,~\cite{tanana1975optimality}) that Ivanov and Morozov
regularization are related to different types of prior knowledge on
the exact equation $F(x^{\text{exact}})=y^{\text{exact}}$. Morozov
regularization is related to prior knowledge about the exact
\emph{data} or the \emph{noise level}, i.e., upper estimates on the
quantity $\rho(y^{\text{exact}},y)$. Ivanov regularization is related
to prior knowledge about the exact \emph{solution}, i.e, about upper
estimates about the quantity $R(x^{\text{exact}})$.  Hence, the choice
between Morozov and Ivanov regularization should be based upon the
available prior knowledge at hand.  However, there are more factors,
which should be taken into account when choosing the variational
method: Since the three optimization
problems~(\ref{eq:tikhonov}),~(\ref{eq:ivanov}), and~(\ref{eq:morozov}) may
belong to different ``subclasses'' of optimization problems their
solution may have different computational complexity.
\begin{example}[Linear problems in Hilbert space]
  In this classical setting, $X$ and $Y$ are Hilbert spaces, $F$ is
  bounded and linear and we use $\rho(Fx,y) = \norm[Y]{Fx-y}^2$ and
  $R(x) = \norm[X]{x}^2$. In this case, the Tikhonov problem has an
  explicit solution $x_\alpha = (F^*F + \alpha\id)^{-1}F^* y$ which
  can be treated numerically in several convenient ways (since the
  operator which has to be inverted is self-adjoint and positive
  definite).
  
  However, for both Ivanov and Morozov regularization no closed
  solution exists in general and one usually resorts to solving a
  series of Tikhonov problems, adjusting the parameter $\alpha$ such
  that the Ivanov or Morozov constraint is
  fulfilled~\cite{frommer1999fastcg}.
\end{example}
\begin{example}[Sparse regularization]
  We consider regularization of a linear operator equation $Ku=g$ with
  an operator $K:\ell^2\to Y$ with a Hilbert space $Y$ by means of a
  sparsity
  constraint~\cite{daubechies2003iteratethresh,lorenz2008reglp,grasmair2008sparseregularization}. In
  this setting one works with the discrepancy functional $\rho(Ku,g) =
  \frac{1}{2}\norm[Y]{Ku-g}^2$ and the regularization functional $R(x)
  = \norm[1]{u}$ (extended by $\infty$ if the 1-norm does not exist).
  In this case Tikhonov regularization consists of solving a convex,
  non-smooth, and unconstrained optimization problem (it is a
  non-smooth convex program, however, with additional structure).
  Morozov regularization consists of solving a non-smooth and convex
  optimization problem with a (smooth) convex constraint (and it can
  be cast as a second-order cone-program), and Ivanov regularization
  requires solving a smooth and convex optimization problem with a
  non-smooth convex constraint (it is a quadratic program).
  
  Looking a little bit closer on this classification and the
  properties of $\rho$ and $R$ we observe that Ivanov regularization
  gives in fact the ``easiest'' problem since it obeys a smooth objective
  function and a constraint with a fairly easy structure
  (e.g.~it is easy to
  calculate projections onto the constraint). On the other hand, the Morozov
  problem is ``difficult'' since it involves a non-smooth objective
  over a fairly complicated convex set (in the sense that projections
  onto the set $\norm{Ku-g}\leq\delta$ are costly to calculate).
  Indeed, this rationale is behind the SPGL1
  method~\cite{vandenberg2007spgl1,vandenberg2008paretofrontiertbasispursuit}:
  It replaces the Morozov problem with a sequence of Ivanov problems,
  solving each by a spectral projected gradient method, resulting in 
  one of the fastest methods available for Morozov regularization
  with $\ell^1$ regularization functional.
  
  % It has been observed in different contexts, that the Ivanov approach
  % yields faster algorithms than the Tikhonov approach in this case,
  % e.g.~the basic projected gradient method for Ivanov problems
  % generally outperforms the basic iterative thresholding algorithm for
  % Tikhonov problems~\cite{daubechies2007projgrad}.
\end{example}
  
In conclusion, the choice between the three variational methods should
be based on the available prior knowledge and also on the tractability
and the complexity of the corresponding optimization problem (often
leading to a combination of two methods).

\section{Necessary conditions for Tikhonov schemes}
\label{sec:nec-cond}
In this section we analyze regularization properties of the Tikhonov
method. First we formalize our requirements for a scheme to be
regularizing in the Tikhonov case. As usual we formulate conditions on
existence, stability and convergence of the minimizers,
cf.~\cite{scherzer2009variationalmethods}.
\begin{definition}[Tikhonov regularization scheme]
  A variational scheme $(\rho,R,\seqstruc)$ for $F:(X,\tau_X) \to (Y,\tau_Y)$
  is called \emph{Tikhonov regularization scheme}, if
  the following conditions are  fulfilled:
  \begin{itemize}
    \item[(R1)] Existence: For all $\alpha>0$ and all $y\in Y$ it holds that $\argmin_{x\in X} T_{\alpha,y}(x)\neq\emptyset$.
    \item[(R2)] Stability: Let $\alpha >0$ be fixed, $y_n\conv{\seqstruc}y$ and 
                      $x_n\in \argmin_{x\in X} T_{\alpha, y_n}(x)$. Then $(x_n)$ converges 
                      subsequentially in $\tau_X$ and for each
                      subsequential limit $\bar{x}$ of $(x_n)$ it holds that
                      $\bar{x}\in\argmin_{x\in X} T_{\alpha, y}(x)$ .
    \item[(R3)] Convergence:  Let $F(x)=y$ have an exact solution $x^{exact}$ such that $R(x^{exact})<\infty$ and
      $y_n\conv{\seqstruc}y$. Then there exists a sequence $(\alpha_n)_n$ of
      positive real numbers such that
     $x_n\in\argmin_{x\in X} T_{\alpha_n,y_n}(x)$ converges subsequentially in $\tau_X$ and
     every subsequential limit $\bar{x}$ is a $\rho$-generalized $R$-minimal solution of
     $F(x)=y$. 
  \end{itemize}
\end{definition}

\subsection{Trivial necessary conditions}
First we list fairly obvious necessary conditions to be regularizing
in the Tikhonov sense.
To that end, we introduce the solution operator
\[
  \begin{array}{rccl}
    \mathcal{A}:&Y\times{]}0,\infty{[} & \rightarrow &2^X \\
    &(y,\alpha) & \mapsto & \argmin_{x\in X}T_{\alpha,y}(x).    
  \end{array}
\]
for the Tikhonov problem~(\ref{eq:tikhonov}).  For fixed $\alpha>0$ we
denote $\mathcal{A}_\alpha(y) = \mathcal{A}(y,\alpha)$. We consider
$\mathcal{A}$ and $\mathcal{A}_\alpha$ as set valued mappings and use
the respective notation (see,
e.g.,\cite{schirotzek2007nonsmoothanalysis}), especially the notion of
the domain $\dom \mathcal{A}_\alpha = \set{y\in
  Y}{\mathcal{A}_\alpha(y)\neq\emptyset}$ and the graph
$\graph(\mathcal{A}_\alpha) = \set{(y,x)\in Y\times
  X}{x\in\mathcal{A}_\alpha(y)}$.

Moreover, we recall that a topology is called \emph{sequential} if it
can be described by sequences, i.e. every sequentially closed set is
closed.

\begin{remark}Let $(\rho,R,\seqstruc)$ be a variational scheme for
  $F:(X,\tau_X)\to (Y,\tau_Y)$. Then obviously (R1) is fulfilled if and only if 
$\dom{\mathcal{A}} = Y\times]0,\infty[$, which implies that $\dom R\cap F^{-1}(\dom \rho(\cdot,y))\neq\emptyset$ and hence
$\range F\cap \dom \rho(\cdot,y)\neq\emptyset$ does hold for all $y\in Y$.
\end{remark}
\begin{theorem} Let $(\rho,R,\seqstruc)$ be a variational scheme for
  $F:(X,\tau_X)\to (Y,\tau_Y)$ that fulfills (R2), $\alpha >0$ and $y\in Y$. Then 
\begin{enumerate}
\item $\mathcal{A}_\alpha(y)$ is sequentially compact and
so is %$\set{x_n}{x_n\in \mathcal{A}_\alpha(y_n)}\cup\mathcal{A}_\alpha(y)$
$\left(\bigcup_{n\in\NN}\mathcal{A}_\alpha(y_n)\right)\cup\mathcal{A}_\alpha(y)$
 for every sequence $(y_n)$ in $Y$ such that $y_n\conv{\seqstruc}y$.
\item The implication\[\left.
    \begin{array}{c}
     y_n\conv{\seqstruc}y\\
     x_n\conv{\tau_X}x\\
     x_n\in\mathcal{A}(y_n,\alpha)
    \end{array}
   \right\rbrace\Rightarrow x\in \mathcal{A}(y,\alpha)
 \]
  does hold,
  i.e. the mapping $\mathcal{A}_\alpha$
  is sequentially closed.
  
  If $\seqstruc$ is induced by a topology $\tau$ and $\tau\times \tau_X$ is 
  sequential, %es gibt leider sequ. Raeume, dass sogar deren kartesisches Quadrat nicht sequ. ist
  then $\graph(\mathcal{A}_\alpha)$ is closed for every $\alpha >0$.
  
 If furthermore $\mathcal{A}_\alpha$ is single valued, then (R2) does hold if and only if $\mathcal{A}_\alpha$ is continuous w.r.t. $\seqstruc$ and the
  sequential convergence structure of $\tau_X$.
  % \rem{Irgendwo hinschreiben was Stetigkeit zwischen sequentiellen Konvergenzstrukturen
  % überhaupt heißt! Umformulierung, weil: es ist nicht die Stetigkeit der mengenwertigen Abbildung gemeint, die beiden einfach
  % miteinander identifiziert. Stetigkeit von mengenwertigen Abb. schon im top. Fall schwierig, für KS müsste man das vorher explizit
  % definieren, was damit gemeint ist. Bei Stetigkeit impliziert (R2) evtl Probleme, falls GW nicht eindeutig}
\end{enumerate}
\end{theorem}
\begin{proof}
	\begin{enumerate}
       \item Let $(x_n)$ be a sequence in $\mathcal{A}_\alpha(y)$ and consider the constant sequence $y_n:=y$. Then 
  $y_n\conv{\seqstruc}y$ and $x_n\in\mathcal{A}_\alpha(y_n)$ do hold. Therefore (R2) implies the existence of
  a convergent subsequence of $(x_n)$ converging to an element of $\mathcal{A}_\alpha(y)$.

  To prove the second assertion, let  $(x_k)_{k\in\NN}$ be a sequence in 
  $\left(\bigcup_{n\in\NN}\mathcal{A}_\alpha(y_n)\right)\cup\mathcal{A}_\alpha(y)$.   We distinguish two cases:
  \begin{enumerate}
  	\item There exists a $\tilde{y}\in\set{y_n}{n\in\NN}\cup\{y\}$ such that $x_k\in\mathcal{A}_\alpha(\tilde{y})$ for 
  	infinitely many $k\in\NN$.
  	%$\set{x_k}{k\in \NN}\cap\mathcal{A}_\alpha(\tilde{y})$ is infinite. 
  	Then $(x_k)$ has a subsequence in $\mathcal{A}_\alpha(\tilde{y})$
  	and the assertion is covered by the first part of the proof.
  	\item For every $\tilde{y}\in\set{y_n}{n\in\NN}\cup\{y\}$ there are at most finitely many $k\in\NN$ such that 
  	$x_k\in\mathcal{A}_\alpha(\tilde{y})$.
  	           %$\set{x_k}{k\in \NN}\cap\mathcal{A}_\alpha(\tilde{y})$ is finite for all $\tilde{y}\in\set{y_n}{n\in\NN}\cup\{y\}$. 
  	           Without loss of generality we can assume that $x_k\not\in\mathcal{A}(y)$ for all $k\in\NN$, that the $x_k$ are pairwise distinct and that there is at most one $x_k\in \mathcal{A}(y_n)$ for all $n\in\NN$.
  	          % We assume without loss of generality that the cardinality of these intersections is at most one 
              %and that $\set{x_k}{k\in \NN}\cap\mathcal{A}_\alpha(y)=\emptyset$ 
              (otherwise we could choose  an appropriate subsequence).
  	          % Let $(\tilde{y}_k)$ be the sequence in $\set{y_n}{n\in \NN}$ given by $x_k\in\mathcal{A}_\alpha(\tilde{y_k})$.
  	           
  	           Then, the sequence given by $\tilde{y}_k:=y_n$ if $x_k\in \mathcal{A}(y_n)$ is well defined and 
  	           $\set{\tilde{y}_k}{k\in\NN}$ is an infinite subset of $\set{y_n}{n\in\NN}$. Hence $(y_n)$ and $(\tilde{y}_k)$
  	           have a subsequence $(\tilde{y}_{k_m})$ in common.
  	          
%   	           If $(\tilde{y}_k)$ had a subsequence $(\tilde{y}_{k_m})$ in common with $(y_n)$ then (R2) would imply that the corresponding 
%   	           subsequence of $(x_k)$ had a convergent subsequence with limit in $\mathcal{A}_\alpha(y)$.
%   	           
%   	           Therefore it is sufficient to construct such a subsequence to show our assertion. Denote by $M$ the underlying set of 
%   	           $(\tilde{y}_k)$. Since $M$ consists of members of $(y_n)$ and $M$ is infinite there exists a subsequence $(y_{n_l})$ of $(y_n)$ which contains every element of $M$ exactly once.
  	           % Since neither $(y_{n_l}$ nor $(\tilde{y}_k)$
  	         %  have a constant subsequence (\rem{check!}) they have a subsequence in common, which we denote by $(\tilde{y}_{k_m})$
  	           %\rem{(Praezise genug formuliert? Vorhandenen Beweis dafuer aufschreiben?)}.
  	           
%   	         Now set $\tilde{y}_{k_1}:=\tilde{y_1}$. 
%   	         For $m>1$ suppose that $\tilde{y}_{k_1},\ldots, \tilde{y}_{k_{m-1}}$ are already defined. Due to the choice of $(y_{n_l})$ there exists an index  $l_{m-1}$ such that $\tilde{y}_{k_{m-1}}=y_{n_{m-1}}$. Set
%   	          \[M_m:=M\setminus\{ \tilde{y}_1,\ldots,\tilde{y}_{k_{m-1} },y_{n_1}\ldots,y_{n_{l_{m-1}}}\}\]      
%   	          and choose $\tilde{y}_{k_m}:=\tilde{y}_{k^*}$ where $k^*=\min\set{k}{\tilde{y}_k\in M_m}$. Due to construction 
%   	          the elements stay in the same order as in both of the original sequences and hence $(y_{k_m})$ is a common subsequence
%   	          of $(\tilde{y}_k)$ and $(y_n)$.
              Now $(\tilde{y}_{k_m})$ being a subsequence of $(y_n)$ implies $\tilde{y}_{k_m}\conv{\seqstruc}y$ and due to construction
              $x_{k_m}\in\mathcal{A}_\alpha(\tilde{y}_{k_m})=\argmin_{x\in X}T_{\alpha,\tilde{y}_{k_m}}(x)$ does hold. Applying (R2) yields
              the existence of a convergent subsequence of $(x_{k_m})$ with limit in $\mathcal{A}_\alpha(y)$, which completes the proof.
  \end{enumerate}
  \item The first assertion is just a reformulation of (R2). As to the second assertion, in the case of topological convergence 
        the sequential closedness of $\mathcal{A}_\alpha$ is equivalent to sequential closedness of $\graph(\mathcal{A}_\alpha)$ w.r.t
        $\tau\times\tau_X$. Since the latter topology is sequential, the assertion follows.

  \end{enumerate}
\end{proof}

\subsection{A closer look on the data space}
  % A closer look on some usual sufficient conditions

% Afterwards we investigate some of the sufficient conditions
% previously assumed to show regularization
% properties~\cite{poeschl2008diss,flemming2010nonmetric,flemming2011diss}
% of $\mathcal{M}$, and analyze them with respect to their
% consequences on the data space $Y$ and leading to necessary
% conditions on the interplay of the convergence notions and the
% discrepancy functional $\rho$. \rem{Der letzte Teilsatz stimmt so
%   nicht ganz} In a relatively specialized setting, this will earn us
% some necessary conditions on the discrepancy functional to be
% covered by the results in REF Poeschl and concrete descriptions of
% two topologies that could describe sequential convergence in the
% data space. Also we gain a better understanding of the interplay
% between some of the usual assumptions.

There exists a vast amount of settings that provide
sufficient conditions for a Tikhonov scheme with non-metric
discrepancy term to be regularizing.
Here we start from a theorem which is extracted
from~\cite{poeschl2008diss,flemming2010nonmetric,flemming2011diss}.
\begin{theorem}\label{thm:poeschl_flemming_example}

 Let $\mathcal{M}=(\rho,R,\seqstruc)$ be a variational scheme for a continuous mapping
  $F:(X,\tau_X)\to (Y,\tau_Y)$ that fulfills the following list of assumptions:
 
 \begin{enumerate}[label=(A\arabic{*})]
     \item \label{it:sublev}The sublevelsets $\set{x\in X}{R(x)\leq M}$ are sequentially compact w.r.t $\tau_X$
              for all $M>0$, so in particular $R$ is sequentially lower semicontinuous%Compactness: Ex.,Stab. Beschr.,Conv. Beschr.; uhs:Ex., Stab. Min, Conv. R-Min 
     \item\label{it:domT} $\dom T_{\alpha, y}\neq \emptyset$ for all $y\in Y$%Ex. nur fuer ein y, bei den restlichen
                %sequ. closed ueberfluessig, wird durch uhs-Forderungen abgedeckt
     \item\label{it:seqcont} $(x,y)\mapsto \rho(F(x),y)$ is sequentially $\tau_X\times \tau_Y$ lower semi continuous %Ex.,
                %Stab. Minimalitaet, Conv. Lsg
    % \item $R$ lsc?????? \rem{CHECK! Wahrscheinlich schon durch die lsc-Forderung an T.. abgedeckt}
%      \item $\rho(z,y)=0$ implies $z=y$ \rem{Wirklich notwendig? Das wird doch nur gebraucht
%                um zu zeigen, dass Lsg. gegen echte lsgn von Fx=y konvergieren? - und um 
%                sicherzustellen, dass eine Topologie existiert, sodass eine Folge konvergiert gdw
%                rho gegen Null}
     \item\label{it:seqstruc} The sequential convergence structure $\mathcal{S}$ is given by 
     
           $y_n\conv{\seqstruc}y$ if and only if  $\rho(y,y_n)\rightarrow 0$ \hfill [CONV]
     
              and furthermore it fulfills
              
              $ y_n\conv{\seqstruc}y$ implies $\rho(z,y_n)\rightarrow \rho(z,y)$ for all
              $z\in \dom\rho(\placeholder, y)$\hfill[CONT]
              %[CONV]:Stab. Min., [CONT]:Stab. Beschr.(hier wuerde beschraenkt reichen)+Min.
     \item\label{it:seqstructau}$y_n\conv{\seqstruc}y$ implies $y_n\conv{\tau_Y}y$ % Stab. Min., Conv. Lsg.
%      \item Let $(/alpha_n)_{n\in\NN}$ be a sequence in $\RR_{>0}$
%            such that $\alpha_n\rightarrow 0$ and $\frac{\rho(y,y_n)}{\alpha_n}\rightarrow 0$ as
%            $y_n\conv{\secstruc}y$.
 \end{enumerate}
 %\rem{Hier muss irgendwo noch etwas zu sequ. closed gesagt werden. Falls GW nicht eindeutig,
 %folgt das nicht aus sequ. Kompaktheit. Aber aus R uhs?- Tut es!}
  Then $\mathcal{M}$ is a Tikhonov regularization scheme.
 \end{theorem}
 \begin{proof} We will only give a sketch of the proof, for details we refer to
  \cite{hofmann2007convtikban,poeschl2008diss, flemming2011diss}.

  Since $\rho$ and $R$ are nonnegative~\ref{it:sublev}--\ref{it:seqcont} imply (R1) (existence of minimizers).
  
  Let $(x_n)$ be a sequence of minimizers as in (R2). Then $(R(x_n))$ is bounded due to~\ref{it:domT} and [CONT]. Hence,~\ref{it:sublev} delivers a convergent subsequence. Let $\bar{x}$ be the limit of such a
  subsequence. Then, \ref{it:seqstructau},~\ref{it:seqcont} and [CONT] yield $T_{\alpha, y}(\bar{x})\leq T_{\alpha,y}(x)$ for all $x\in X$. Consequently, (R2) is fulfilled (stability).

 Let $F(x^\dagger)=y$, $R(x^\dagger)<\infty$ and $(y_n)$ be a sequence such that $y_n\conv{\seqstruc}y$. 
 Then, due to [CONV], there exists $\alpha_n$ such that 
\begin{equation}\label{equ:conv_slow_enough}
  \alpha_n\rightarrow 0 \text{ and } \frac{\rho(y,y_n)}{\alpha_n}\rightarrow 0 \text{ as } n\rightarrow \infty
\end{equation}
 does hold (e.g.~$\alpha_n = \sqrt{\rho(y,y_n)}$).
 
 Therefore $R(x_n)\leq \frac{1}{\alpha_n}T_{\alpha_n,y_n}(x^\dagger)$ for $x_n\in\argmin_{x\in X} T_{\alpha_n,y_n}(x)$ and together with~\ref{it:sublev} this yields subsequential convergence and $R(\bar{x})\leq
 R(x^\dagger)$ for every subsequential limit $\bar{x}$. Using [CONV] we get $\rho(F(x_n),y_n)\rightarrow 0$, which yields $\rho(F(\bar{x}),y)=0$ due to \ref{it:seqstructau} and~\ref{it:seqcont}.
\end{proof}
\begin{remark}
 In \cite{poeschl2008diss} it is additionally assumed that  $\rho(z,y)=0$ implies $y=z$. This allows to formulate (R3) with $R$-minimal solutions in the strict sense (i.e.~with $F(x)=y$) instead of $\rho$-generalized $R$-minimal solutions.
 
 In item~\ref{it:seqstruc} it would be sufficient if [CONT] only holds for $z\in \dom\rho(\placeholder, y)\cap F(X)$.
\end{remark}
 
As remarked earlier, it is hard to obtain necessary conditions for a
general Tikhonov scheme to be regularizing. Hence, we have chosen to start
with the analysis of the data space $Y$. This is motivated by the fact
that there are three different objects that pose additional structure
on $Y$, namely the topology $\tau_Y$, the sequential convergence
structure $\seqstruc$ and the discrepancy functional
$\rho$. Obviously, not every combination of these three objects will lead
to a regularization scheme. We start from
Theorem~\ref{thm:poeschl_flemming_example} and the
conditions~[CONV],~[CONT] and~\ref{it:seqstructau} and investigate
the interplay of $\tau_Y$, $\seqstruc$ and $\rho$ and deduce necessary
conditions on their relations. We are aware that the
conditions~[CONV],~[CONT] and~\ref{it:seqstructau} are not
necessary for a scheme to be regularizing, but they appear as natural
conditions in the context of regularization.  However, we will get two
different topologies whose convergent sequences, under appropriate circumstances, come naturally to fulfill one of the conditions 
[CONV] and [CONT], respectively,
%provide exactly the desired convergent
%sequences\rem{Dieser Satz/Abschnitt erzeugt eine falsche Erwartungshaltung, mißverständlich, evtl. durch einen Halbsatz entschärfen},
both given in a constructive way. Moreover, they will provide means to analyze other topologies having the desired convergent sequences (see Remark
\ref{rem:we_need_Z_and_tilde_Y} for details). Applied to specific
classes of discrepancy functionals this could allow a deeper
structural insight on what [CONT] does really mean and may tackle a
subclass for which Theorem \ref{thm:poeschl_flemming_example} is
eligible without further adaptions.
  
% % \rem{Definition: Nur hier lokal, um knapp darauf Bezug nehmen zu koennen oder Namen schon oben vergeben?}
% \begin{definition} Let $(Y, \tau_Y)$ be a topological space and
% $\rho:Y\times Y\rightarrow [0,\infty]$ a functional such that $\rho(y,y)=0$ for all $y\in Y$.
% We call the triple $(Y, \tau_Y, \rho)$ a \emph{compatible data space model} if the following conditions hold for
% all sequencences $(y_n)_{n\in\NN}$ in $Y$:

% \begin{itemize}
%  \item $y_n\conv{\seqstruc}y$ if and only if $\rho(y,y_n)\rightarrow 0$\hfill [CONV]
%  \item $y_n\conv{\seqstruc}y$ and $ \rho(z,y)<\infty$ implies $ \rho(z,y_n)\rightarrow \rho(z,y)$\hfill [CONT]
% \end{itemize}
% \rem{Konvergenz gegen Null von Rho folgt aus CONT und dass rho auf der Diagonalen verschwindet! Die eine Richtung ist also redundant}
% \end{definition}

\begin{remark}
  Every topology $\tau$ induces a sequential convergence structure $\seqstruc(\tau)$ via $y_n\conv{\seqstruc(\tau)} y$ if
  and only if $y_n\conv{\tau}y$.  
  In the further course of the paper we will say, that a topology $\tau$ satisfies [CONV] respectively [CONT]
  if and only if the sequential convergence structure induced by the topology has the respective 
  property.
\end{remark}

Now we define the two topologies mentioned above, the first one designed to satisfy [CONV],
the second to satisfy [CONT].

\begin{definition}
  \label{def:taurho-tauin}
  Let $Y$ be a set and $\rho:Y\times Y\rightarrow [0,\infty]$ such that 
  $\rho(y,y)=0$ for all $y\in Y$.
  \begin{enumerate}
  \item We call 
  \[
        \ball^\rho_\varepsilon(z):=\set{y\in Y}{\rho(z,y)<\varepsilon}
  \]
  the $\varepsilon$-ball w.r.t $\rho$ centered at $z$ and set
  \[
     \tau_\rho:=\set{U\subseteq Y}{\forall z\in U\,\exists \,\varepsilon>0\text{ such that }\ball^\rho_\varepsilon(z)\subseteq U }\,.
     \]
     \item Let $Z\subseteq Y$ and  $\tilde{Y}\subseteq Y$ and let $[0,\infty]$ be equipped with the one-point compactification of the standard
     topology on $[0,\infty[$. For $z\in Z$ we define 
              \[
                     f_z:\tilde{Y}\rightarrow [0,\infty] \text{ by } \tilde{y}\mapsto \rho(z,\tilde{y})\,.
              \]
              By $\topin$ we denote the initial topology on $\tilde{Y}$ w.r.t the family $(f_z)_{z\in Z}$
                i.e. the coarsest topology on $\tilde{Y}$ for which all the $f_z$ are continuous.
                
   \end{enumerate}
\end{definition}
Note that the notation $\topin$ does not reflect the dependency on
$\tilde Y$ and $Z$. Hence, throughout the paper we will always mention
explicitly the involved $\tilde{Y}$ and $Z$.

\begin{remark}\label{rem:we_need_Z_and_tilde_Y}
  The two additional sets $Z$ and $\tilde Y$ are introduced to allow
  to model a broader class of discrepancy functionals and to construct
  a larger variety of topologies. First, note that there are
  non-symmetric discrepancy functionals and even ones in which the
  domains of $\rho(\cdot,y)$ and $\rho(z,\cdot)$ differ. Especially,
  both arguments of $\rho$ have different meanings: The first argument
  takes images of solutions $x$ under $F$ which can have additional
  structure (e.g.~due to discretization), while the second argument
  takes measured data which may also have additional
  characteristics. Moreover, a smaller $Z$ will allow for a coarser
  topology (and this will be helpful if the range of $F$ is a
  ``small'' set) and a smaller $\tilde Y$ can model only a restrictive
  set of possible data (e.g. strictly non-negative one).
  This purpose could also be met by restricting ourselves to $\rho:Z\times\tilde{Y}\rightarrow [0,\infty]$ and
  $F:X\rightarrow Z$ for appropriately chosen $Z$ and $\tilde{Y}$ in the first place.
%   If $Z=\tilde{Y}=Y$ then $\topin$ does satisfy [CONT]. Note that continuity of all the $f_z$ on the whole 
%   of $Y$ is stronger than required in [CONT]. There we only need continuity of $f_z$ in $Y$ for all
%   $z\in \dom\rho(\placeholder, y)$.
%   

  The reason for not doing so is, that
  the topology $\topin$ is not merely designed to be itself a possible member of a 
          regularization scheme but also as a tool to analyze other topologies that induce a
          convergence structure as in Theorem 3.4 (A4).
          
          The original aim of constructing the two topologies was to derive conditions on topologies on the whole of $Y$ whose 
          sequential convergence structures fulfill both conditions demanded in Theorem 3.4 (A4) by sandwiching them between
          $\tau_\rho$ and $\topin$. In favour of this purpose we want $\tau_\rho$ to be some sort
          of maximal topology fulfilling [CONV] and $\topin$ to be minimal with the property [CONT].
          
          As we will see in Theorem \ref{thm:props_tau_on_Y} the first request is achieved easily if 
          there is any topology satisfying [CONV], the second ambition 
          is a little bit more complicated. If we took the initial topology w.r.t the familiy 
          $\{\rho(z,\cdot)\mid z\in Y\}$ we certainly would satisfy [CONT], but we would loose minimality as 
          soon as there is a $y\in Y$ with $\dom\rho(\cdot,y)\subsetneq Y$. Therefore we have to choose
          a index set $Z\subsetneq Y$ in this case to avoid more continuous $\rho(z,\cdot)$ than required by
          [CONT] and hence $\tau_{IN}$ too fine. As to $\tilde{Y}$: If there exist $y_1$, $y_2\in Y$ such that 
          $\dom\rho(\cdot, y_1)\neq \dom\rho(\cdot, y_2)$ there will not exist a $Z\subset Y$
          such that the convergence condition from [CONT] is only fulfilled for $z$ which also satisfy the finiteness
          condition. We can cure this by choosing $Z$ smaller as long as the intersection of all these domains is 
          nonempty, otherwise choosing $\tilde{Y}\subsetneq Y$ may  allow to apply the analysis at least to topological
          subspaces.
          
          Also this approach is not carried out to its full extent, we would like to leave the way open to do so.
\end{remark}

\begin{example}[Metrics and powers of norms]
  If $\rho$ be a metric on $Y$, $Z=\tilde Y=Y$ and $\seqstruc$ defined as in [CONV].  Then
  $\seqstruc$ and $\rho$ satisfy [CONT]: Let $\rho(y,y_n)\rightarrow
  0$ and $z\in Y$. Then
  \[
  \rho(z,y_n)\leq \rho(z,y)+\rho(y,y_n)\rightarrow \rho(z,y)\,.
  \]
  Clearly, the triangle inequality could be replaced by a
  quasi-triangle inequality and hence, for the popular case of
  $\rho(z,y) = \norm[Y]{z-y}^p$ of the $p$-th power of a norm (with
  $p>0$) a similar claim is valid: [CONV] says that $\seqstruc$ is the
  norm convergence (independent of the value of $p$), and hence,
  [CONT] is fulfilled.
  
  Moreover, in the cases of (quasi-)metrics and positive powers of
  norms the topologies $\tau_\rho$ and $\tau_{IN}$ coincide and are the
  metric or norm topology, respectively.
\end{example}

We may caution the reader that in general the two topologies
$\tau_\rho$ and $\tau_{IN}$ may be different, as is shown by the
following (somewhat pathological) example:
\begin{example}
  \label{ex:tauIN-taurho-pathological}
  In $Y = \tilde Y = Z = \RR^2$ consider
  \[
  \rho(z,y) = \left\{
    \begin{array}{ll}
      0 & \#\{i\ :\ y_i\neq z_i\}\leq 1\\
      1 & \text{else}.
    \end{array}\right.
  \]
  In other words, two elements are considered equal if they differ
  only in one coordinate. In this case one can show that the topology
  $\tau_\rho$ is the indiscrete topology (i.e. the only open sets are
  $Y$ and the empty set). However, $\topin$ has the subbasis
  \[
  B_{1/2}(z) = \{y\in \RR^2\ :\ \#\{i\ :\ y_i\neq z_i\}\leq 1\}
  \]
  and hence $\topin$ is finer than $\tau_\rho$. Moreover, $\topin$ has
  less convergent sequences than $\tau_\rho$ (in which every sequence
  converges to every point).
\end{example}

For the reader's convenience we recall some properties of the topologies $\tau_\rho$ and $\topin$ that will be used in the further course of the paper.
\begin{lemma}
 The following properties hold for $\tau_\rho$:
 %\rem{stimmt das ueberhaupt noch, wenn unendlich zugelassen? - Ja!}
 \begin{enumerate}[label=(\roman{*})]
 \item\label{lem:prop_top1} $\tau_\rho$ is a sequential topology.
  \item\label{lem:prop_top2} A mapping from $Y$ to an arbitrary topological space is $\tau_\rho$-continuous if and only if
          it is sequentially continuous w.r.t $\tau_\rho$.
    \item\label{lem:prop_top3} $\rho(y,y_n)\rightarrow 0$ implies $y_n\conv{\tau_\rho} y$.
 \end{enumerate}
 The following holds for $\topin$:
 \begin{enumerate}[label=(\roman{*})]\setcounter{enumi}{3}
 \item\label{lem:prop_top4} For arbitrary $Z,\tilde{Y}\subseteq Y$ sequential convergence w.r.t $\topin$ can be
   characterized as follows:
   
   Let $(y_n)_{n\in \NN}$ be a sequence in $\tilde{Y}$ and $y\in
   \tilde{Y}$.  Then $y_n\conv{\topin} y $ if and only if
   $\rho(z,y_n)\rightarrow \rho(z,y)$ for all $z\in Z$.
   
 \item\label{lem:prop_top5} If additionally $\tilde{Y}\subseteq Z$ does hold,
   $y_n\conv{\topin} y $ implies $\rho(y,y_n)\rightarrow 0$.
 \end{enumerate}
\end{lemma}
\begin{proof}
  For \textit{\ref{lem:prop_top1}}~see \cite[\S 2.4]{Arkhangelski1990basic.concepts.general.topology}.
  Item \textit{\ref{lem:prop_top2}}~is a direct consequence of $\tau_\rho$ being sequential and
  \textit{\ref{lem:prop_top3}}~is clear from the definition of open sets w.r.t.~$\tau_\rho$. Then, the first implication of
  \textit{\ref{lem:prop_top4}}~is due to the sequential continuity of continuous maps and the converse holds because the set $\set{f_z^{-1}(V)}{z\in Z,\,V\subseteq [0,\infty]\text{ open}}$ is a subbase for $\topin$. Finally, \textit{\ref{lem:prop_top5}}~is the continuity of $f_y$ at $y$.
\end{proof}

Now we investigate the relation of $\tau_\rho$ to the property [CONV].
\begin{theorem}
  \label{thm:props_tau_on_Y}
 Let $\tau$ be a topology on $Y$. Then the following does hold:
 \begin{enumerate}
   \item The property
          \begin{equation}\label{equ:conv_impl}
            \rho(y,y_n)\rightarrow 0 \text{ implies } y_n\conv{\tau}y
            \end{equation}
             does hold if and only if
             $\tau$ is coarser than $\tau_\rho$.
   \item If $\tau$ has property [CONV], then so does $\tau_\rho$. In particular
            $\tau_\rho$ is the finest topology with that property.
 \end{enumerate}
%If exists $\tau_Y$ with [CONV] than $\tau_\rho$ fulfills [CONV] and 
%$\tau_\rho$ finer than $\tau_Y$
\end{theorem}
\begin{proof}
  \begin{enumerate}
     \item Let $\tau$ be coarser than $\tau_\rho$, then every $\tau_\rho$-convergent
             sequence is also $\tau$-convergent, and therefore (\ref {equ:conv_impl}) does
             hold.
             
             Now let $\tau$ be a topology where  (\ref {equ:conv_impl}) does hold. Suppose there exists $U\in\tau$ and $U\not\in\tau_\rho$. Then there is an $u\in U$ such that for all $n\in \NN$ there exists a $y_n\in\ball_{\frac{1}{n}}(u)\setminus U$. Evidently $\rho(u,y_n)\rightarrow 0$ does hold and since (\ref{equ:conv_impl}) does hold w.r.t $\tau$, this implies $y_n\conv{\tau}u$ in contradiction to $y_n\not\in U$ for all $n\in\NN$.
      \item Let $\tau_Y$ be a topology that fulfills [CONV] and $(y_n)$ a $\tau_\rho$ convergent sequence with limit $y$. Due to (i)~$\tau$ is coarser than $\tau_\rho$, therefore
      $y_n\conv{\tau_Y} y$ and consequently $\rho(y,y_n)\rightarrow 0$.
  \end{enumerate}
\end{proof}
So, if $\mathcal{S}$ is induced by a topology at all, this is also done by the relatively well-behaved
(i.e. sequential) topology $\tau_\rho$.

One case in which this applies is marked out in the following remark:
\begin{remark}\label{rem:topological_if_unique_limits}
 If $\seqstruc$ provides unique limits, it is induced by a topology: Because
a sequence $\mathcal{S}$-converges given all its subsequences have a subsequence tending to the same limit, this is guaranteed e.g. by
\cite[Prop.~1.7.15]{beattie2002convergence}, \cite{kisynyki1959convergencedutypeL}.
\end{remark}

We have two further immediate consequences of Theorem~\ref{thm:props_tau_on_Y}:
\begin{corollary}
  Let (A4) of Theorem~\ref{thm:poeschl_flemming_example} hold. 
  Then (A5) of Theorem~\ref{thm:poeschl_flemming_example} does hold for a topology 
  $\tau_Y$ on $Y$ if and only if
  $\tau_Y$ is coarser than $\tau_\rho$.
\end{corollary}
In the case of $\rho(z,y) = \norm{z-y}^p$ of Banach space norm, this
means that $\tau_Y$ is coarser than the norm topology, i.e. this
condition which has been required previously
(cf.~\cite{hofmann2007convtikban}) is somehow necessitated.
\begin{corollary} If there is a topology $\tau$ where $\rho(y,y_n)\rightarrow 0$ implies
  $y_n\conv{\tau}y$ such that [CONT] is fulfilled, then $\tau_\rho$ also fulfills
  [CONT].
\end{corollary}
Since we are only interested in sequential convergence, this allows us to take $\tau_\rho$
as a sort of model topology.
\begin{remark} In general, the set $\tau_\seqstruc$ of all sequentially open sets w.r.t to a sequential convergence structure $\seqstruc$ on $Y$ is a topology on $Y$.
   As has been shown in \cite[Prop. 2.10]{flemming2011diss}, in the case that $\seqstruc$ is given by [CONV], it is sufficient for [CONV] to hold for the topology $\tau_\seqstruc$ as well, that $\seqstruc$ fulfills [CONT].
   
   Therefore assumption \ref{it:seqstruc} implies that $\tau_\rho$ also has [CONV] and this again
   implies that $\tau_\mathcal{S}=\tau_\rho$, since $\tau_\rho$ is sequential. Moreover, in this case the sets  $\ball^\rho_\varepsilon(y)$
   are open for all $\varepsilon >0$, $y\in Y$  (see also \cite{flemming2011diss}) and therefore constitute a base for $\tau_\rho$.
\end{remark}

The next theorem deals with the question what consequences it has if [CONT] does
hold in $\tau_\rho$.
\begin{theorem}\label{thm:tau_rho_has_CONV} Let $Z\subseteq\bigcap_{y\in Y}\dom\rho(\placeholder,y)$
  be nonempty and $\tilde{Y}=Y$.%(\rem{alt:$\dom\rho=Y\times Y$}.\rem{Alternativ: Bed. an Bez. zwischen $Z$ etc.abstrakt formulieren,
 %das unten als FOlgerung}
 
 If $\tau_\rho$ fulfills [CONT] then the following does hold: 
  \begin{enumerate}
    \item$\topin$ is coarser than $\tau_\rho$ 
    \item If $Z=Y$ then $\tau_\rho$ and $\topin$ both satisfy [CONV]. In particular they have the same convergent sequences.
    %\rem{Hier drauf achten, was mit unendlich passiert - nichts, weil die Bed. oben sagt, dass unendlich nicht als Wert vorkommt}
   \end{enumerate}
\end{theorem}
\begin{proof}\ 

 \begin{enumerate}
   \item Since $\rho(z,\placeholder)$ is sequentially continuous for all $z\in Z$, it is also 
            continuous and therefore $\topin$ is coarser than $\tau_\rho$.
    \item Due to (i)~convergence w.r.t.~$\tau_\rho$ yields convergence w.r.t..~$\topin$ and hence
             $\rho(y,y_n)\rightarrow 0$ implies $y_n\conv{\topin}y$. Since $Y\subseteq Z$ the 
             converse is also true and therefore $\topin$ satisfies [CONV], and so does $\tau_\rho$.
 \end{enumerate}
\end{proof}

\begin{remark}
 If $\tilde{Y}\subseteq Y$ and $(\tau_\rho)_{\mid\tilde{Y}}$ is
 sequential (e.g. if $\tilde{Y}$ open or closed w.r.t $\tau_\rho$, see \cite{Franklin1965sequencessuffice}), then $(\tau_\rho)_{\mid\tilde{Y}}=\tau_{\rho_{\mid\tilde{Y}}}$ does hold.  In a setting where 
 $\tilde{Y}\subsetneq Z\subseteq Y$, this together with Theorem \ref{thm:tau_rho_has_CONV} would still guarantee, that $\topin$ and the
 subspace topology of $\tau_\rho$ on $\tilde{Y}$ provide the same convergent sequences.

 If $\rho(z,\placeholder)$ is $\tau_\rho$-continuous at every $y\in Y$ for all $z\in Z$ regardless of the finiteness condition in [CONT],
 then we can drop the assumption $Z\subseteq\bigcap_{y\in Y}\dom\rho(\placeholder,y)$ in Theorem \ref{thm:tau_rho_has_CONV}.
\end{remark}
% \begin{corollary}
%   Let $(Y,\tau_Y,\rho)$ be compatible. Then $\tau_\rho$ and $\topin$ have the same 
% convergent sequences as $\tau_Y$.
% \end{corollary}
% \begin{proof}
%   Due to theorem \ref{thm:tau_rho_has_CONV} $\tau_\rho$ and $\tau_Y$ have the same
%   convergent sequences. Since [CONT] regards only sequences, $\tau_\rho$ has [CONT].
% \end{proof}

So, in the setting of Theorem~\ref{thm:tau_rho_has_CONV} sequential
convergence in $\tau_\rho$ and $\topin$ coincides. In general the
sequential convergence structures of these topologies can be different
from each other, cf.~Example~\ref{ex:tauIN-taurho-pathological}.

%\subsubsection*{Bregman distances - an example}
\subsection{Application to Bregman discrepancies}
We conclude Section~\ref{sec:nec-cond} by an application to a special
class of discrepancy functionals, namely ones that stem from Bregman
distances which appear, e.g., in the case of Poisson noise or
multiplicative
noise~\cite{resmerita2005regbanspaces,benning2011errorestimates,luke2012linconvapprox}.
Especially, this gives an example that illustrates how
Theorem~\ref{thm:tau_rho_has_CONV} can be used to gain necessary
conditions on the discrepancy functional for
Theorem~\ref{thm:poeschl_flemming_example} to apply.

Also we treat the
question, when $\rho(y_1,y_2)=0$ implies $y_1=y_2$ in this case.

In the following let $V$ be a Banach space and
$J:V\rightarrow[0,\infty]$ proper, convex, $Z=Y=\dom J$ and
$\tilde{Y}\subseteq\set{y\in Y}{J \text{ has a single valued
    subdifferential at } y}$. The mapping which maps to the unique subgradient of
$J$ is denoted by $\nabla J$.  As distance functional $\rho$ we
consider the Bregman distance w.r.t.~to $J$, i.e., for
$(z,y)\in Y\times \tilde Y$ the functional
\[
D_J(z,y)=J(z)-J(y)-\inner{\nabla J(y)}{z-y}.
\]
\begin{lemma}
  \label{lem:definiteness-DJ}
  Let $y_1,y_2\in \tilde{Y}$. Then $D_J(y_1,y_2)=0$ if and only $\nabla
  J(y_1)=\nabla J(y_2)$. In the case $\tilde Y = V$ the property
  \[
  D_J(y_1,y_2)=0 \Rightarrow y_1=y_2 \text{ for all } y_1,y_2\in V
  \]
  does hold if and only if $J$ is strictly convex.
\end{lemma}
\begin{proof}
  First let $\rho(y_1,y_2)=0$. Then $J(y_1)=J(y_2)+\inner{\nabla
    J(y_2)}{y_1-y_2}$ and hence linearity of $\nabla J(y_2)$ and
  nonnegativity of $\rho$ imply $J(v)-J(y_1)-\inner{\nabla
    J(y_2)}{v-y_1}=\rho(v,y_2)\geq 0$ for all $v\in V$. Therefore
  $\nabla J(y_2)$ is a subgradient of $J$ in $y_1$. Since the subgradient of $J$ is
  single valued at $y_1$ this yields $\nabla J(y_2)=\nabla J(y_1)$.
    
  Now let $\nabla J(y_2)=\nabla J(y_1)$. Then $0\geq
  -\rho(y_1,y_2)=\rho(y_2,y_1)\geq 0$.
  
  For the second statement note that we only need to show that a
  function is strictly convex if and only if the mapping $\nabla J$ is
  injective. Assume that $J$ is strictly convex but that there is
  $\xi=\nabla J(y)=\nabla J(z)$ for $y\neq z$. Plugging $y$ and $z$ in
  the respective (strict) subgradient inequalities gives
  \[
  \begin{array}{rl}
    J(y) - J(z) &>\inner{\xi}{y-z}\\
    J(z) - J(y) &>\inner{\xi}{z-y}
  \end{array}
  \]
  and adding both inequalities we arrive at the contradiction
  $0>0$. Moreover, assuming that $J$ is not strictly convex, there are
  $y\neq z$ such that $J(y)-J(z) = \inner{\nabla J(z)}{y-z}$. But then
  is holds for all $z'$ that $J(z') - J(y) - \inner{\nabla J(z)}{z'-y}
  = J(z') - J(z) - \inner{\nabla J(z)}{z'-z}\geq 0$ which shows that
  $\nabla J(y) = \nabla J(z)$, i.e. that $\nabla J$ can not be
  injective.
\end{proof}

The Bregman distance $D_J$ is in general not symmetric and the
behavior in both coordinates can be quite different (e.g., $D_J$ is
always convex in the first coordinate but not necessarily so for the
second). In the literature, both the discrepancies
\[
\rho_1(F(x),y) = D_J(F(x),y)
\]
and
\[
\rho_2(F(x),y) = D_J(y,F(x))
\]
are used (see~\cite{luke2012linconvapprox} for the
first variant
and~\cite{resmerita2005regbanspaces,benning2011errorestimates} for the
second).

First, we analyze the variant $\rho_1(z,y) = D_J(z,y)$ which corresponds to the Tikhonov function $T_{\alpha,y}(x) = D_J(F(x),y)+\alpha R(x)$.
The following lemma explores how convergence w.r.t $\topin$ actually looks like.
\begin{lemma}
  \label{lem:tauIN_characterization}
  For all sequences $(y_n)$ in $\tilde{Y}$,
  $y\in\tilde{Y}$ the following does hold: $y_n\conv{\topin} y$ if and only if  $\rho_1(y, y_n)\rightarrow 0$ and $\inner{\nabla J(y_n)-\nabla J(y)}{y-z}\rightarrow 0$ for all $z\in Z$.
   %\item
     Moreover, 
     %If $\tilde{Y}$ is a linear subspace of $V$, then
     $y_n\conv{\topin}y$ if and only if 
    $\rho_1(y,y_n)\rightarrow 0$ and $\nabla J(y_n)\wstarto\nabla J(y)$ in $\spann(Z)^\ast$. In particular
    $(\nabla J)_{\mid \tilde{Y}}:\tilde{Y}\rightarrow \spann(Z)^\ast$ is sequentially $\topin$-weak* continuous.
\end{lemma}
\begin{proof}
   The identity
    $ \rho_1(z,y_n)-\rho_1(z,y)=\rho_1(y,y_n)+\inner{\nabla J(y_n)-\nabla J(y)}{y-z}$
    does hold for all $z\in Z$.
    
    So clearly $\rho_1(y,y_n)\rightarrow 0$ and $ \inner{\nabla J(y_n)-\nabla J(y)}{y-z}\rightarrow 0$ imply $y_n\conv{\topin} y$.
    
    Conversely, let $y_n\conv{\topin} y$ hold. Then $\rho_1(y,y_n)\rightarrow 0$ and hence
    $0=\lim_{n\conv{}\infty}(\rho_1(z,y_n)-\rho_1(z,y)-\rho_1(y,y_n))=\lim_{n\conv{}\infty}\inner{\nabla J(y_n)-\nabla J(y)}{y-z}$.
\end{proof}
\begin{corollary}\label{cor:abl_schwachstern_stetig}
  Let $\dom J=\tilde Y = V$ .%$\nabla J$ be injective and 
%  \begin{itemize}
%   \item $\rho(z,y)=0$ if and only if $z=y$. Therefore $\tau_\rho$ does fulfill
%     [CONV] \rem{REF BB oder besser was anderes} Das reicht nicht! GW muessen
%     eindeutig sein, dafuer braucht es schon [CONT]! Dann Injektivitaet wieder
%     unnoetig
  %\item Sequential limits w.r.t $\topin$ unique? Ja, wegen der schwach*-Konvergenz der 
   % Ableitungen \rem{So is das eher Quatsch}
 
  %\item 
  \begin{enumerate}
     \item If $\tau_{\rho_1}$ satisfies [CONT], then $\nabla J$ is $\tau_{\rho_1}$-weak* continuous. 
     \item $\topin$ provides unique sequential limits if and only if $J$ is strictly convex.
  \end{enumerate}
  %\end{itemize}
\end{corollary}
\begin{proof}
  (i) is direct consequence of the previous lemma and (ii) is a direct
  consequence of Lemma~\ref{lem:definiteness-DJ} and the definition of
  $\topin$.
\end{proof}
So, if $J$ is strictly convex, in the setting of  Corollary \ref{cor:abl_schwachstern_stetig} it is necessary for Theorem \ref{thm:poeschl_flemming_example} to apply to Bregman discrepancies that $J$ has $\tau_{\rho_1}$-weak* continuous derivative,
since in this case the sequential convergence structure is given by $\tau_{\rho_1}$ (due to Remark \ref{rem:topological_if_unique_limits}
and Theorem \ref{thm:props_tau_on_Y}).
Moreover, in this case, a Tikhonov regularization scheme with discrepancy $\rho_1$ guarantees convergence 
to an exact solution given $J$ is strictly convex.

To complement Lemma~\ref{lem:tauIN_characterization}, we now analyze the variant $\rho_2(z,y) = D_J(y,z)$ which corresponds to
the Tikhonov functional $T_{\alpha,y}(x) = D_J(y,F(x)) + \alpha R(x)$.
Similarly to Lemma~\ref{lem:tauIN_characterization} we state the
following characterization of convergence with respect to $\topin$.
\begin{lemma}
  \label{lem:tauIN_characterization2}
  For all sequences $(y_n)$ in $\tilde{Y}$, $y\in\tilde{Y}$ the
  following does hold: $y_n\conv{\topin} y$ if and only if $\rho_2(y,
  y_n)\rightarrow 0$ and $\inner{\nabla J(y)-\nabla
    J(z)}{y-y_n}\rightarrow 0$ for all $z\in \tilde Y$
    \end{lemma}
\begin{proof}
  In this case the identity
  $ \rho_2(z,y_n)-\rho_2(z,y)=\rho_2(y,y_n)+\inner{\nabla J(y)-\nabla J(z)}{y-y_n}$
  does hold for all $z\in Z$.  
  So clearly $\rho_2(y,y_n)\rightarrow 0$ and $\inner{\nabla J(y)-\nabla
    J(z)}{y-y_n}\rightarrow 0$ imply $y_n\conv{\topin} y$.
    
  Conversely, let $y_n\conv{\topin} y$ hold. Then
  $\rho_2(y,y_n)\rightarrow 0$ and hence
  $0=\lim_{n\conv{}\infty}(\rho_2(z,y_n)-\rho_2(z,y)-\rho_2(y,y_n))=\lim_{n\conv{}\infty}\inner{\nabla
    J(y)-\nabla J(z)}{y-y_n}$.
\end{proof}

\subsection{The Kullback-Leibler divergence}
\label{sec:kullb-leibl-diverg}

If the noise is modeled by a Poisson process, the appropriate
discrepancy functional is the so-called Kullback-Leibler
divergence~\cite{resmerita2007kullbackleibler,benning2011errorestimates}. This
model fits into the context of Bregman distances and we provide the
setup as in~\cite{resmerita2007kullbackleibler}: Consider a bounded
set $\Omega$ in $\RR^n$ equipped with the Lebesgue measure $\mu$ and
set $V = L^1(\Omega)$. We define
\[
J(y) = \left\{
  \begin{array}{ll}
    \int_\Omega y\log(y)  - y\, d\mu & y\geq 0 \text{ a.e., }\ y\log(y)\in L^1(\Omega)\\
    \infty & \text{else.}
  \end{array}\right. 
\]
From~\cite{resmerita2005regbanspaces} we take the following facts:
The functional $J$ is strictly convex, we have $Z = \dom J = \{y\in L^1(\Omega)\ :\ y\geq 0 \text{ a.e., }\
y\log(y)\in L^1(\Omega)\}$ and it holds that
\[
\partial J(y) = \left\{
  \begin{array}{ll}
    \{\log(y)\}& y\geq \epsilon \text{ a.e. for some $\epsilon>0$,}\ y \in L^\infty(\Omega)\\
    \emptyset & \text{else.}
  \end{array}\right.
\]
Hence, we denote $\nabla J(y) = \log(y)$ and we have
\[
\tilde Y = \{y\in L^1(\Omega)\cap L^\infty(\Omega)\ :\ y\geq
\epsilon \text{ for some } \epsilon>0\}.
\]
The associated Bregman distance is also known as Kullback-Leibler
divergence
\[
D_{KL}(z,y) = \int_\Omega z\log\Big(\frac{z}{y}\Big) - z + y\, d\mu.
\]
From~\cite{borwein1991entropy} it is known that
\[
\norm[1]{z-y}^2\leq \Big(\frac23\norm[1]{y} +
\frac43\norm[1]{z}\Big)D_{KL}(z,y).
\]

\begin{lemma}
  It holds that $\overline{\spann(Z)} = L^1(\Omega)$ and consequently
  $\spann(Z)^*= L^\infty(\Omega)$.
\end{lemma}
\begin{proof}
  Consider $y\in L^1(\Omega)$ and $\epsilon>0$ and define
  \[
  y^\epsilon_+(x) = \left\{
    \begin{array}{ll}
      \frac1\epsilon & y(x)\geq \frac1\epsilon\\
      y(x) & \epsilon<y(x)<\frac1\epsilon\\
      \epsilon & y(x)\leq\epsilon
    \end{array}
  \right., \qquad y^\epsilon_-(x) = \left\{
    \begin{array}{ll}
      -\frac1\epsilon & y(x)\leq -\frac1\epsilon\\
      y(x) & -\frac1\epsilon<y(x)<-\epsilon\\
      -\epsilon & y(x)\geq-\epsilon
    \end{array}
  \right. .
  \]
  Then it holds that $y^\epsilon_+,y^\epsilon_-\in Z$ and
  $(y^\epsilon_+ + y^\epsilon_-)\to y$ in $L^1(\Omega)$. Since every
  continuous linear functional on $\spann(Z)$ can be extended
  continuously to $\overline{\spann(Z)} = L^1(\Omega)$, we conclude
  that $\text{span}(Z)^* = L^\infty(\Omega)$.
\end{proof}

First we look at $\rho_1$ as in Lemma~\ref{lem:tauIN_characterization}: Here
$\rho_1(z,y) = D_{KL}(z,y)$, i.e., the measured data is in the second
argument of the Kullback-Leibler divergence as, e.g.,
in~\cite{benning2011errorestimates,resmerita2007kullbackleibler}. We
deduce directly from Lemma~\ref{lem:tauIN_characterization}: A
sequence $(y_n)$ converges in $\topin$ to $y$ if and only if
\[
\rho_1(y,y_n)  = D_{KL}(y,y_n)\to 0\quad\text{and}\quad \log(y_n)\wstarto \log(y)\
\text{in}\ L^\infty(\Omega).
\]
We can describe this notion of convergence in more familiar terms:
\begin{theorem}
  In the case of $\rho_1(z,y) = D_{KL}(z,y)$ and $\topin$ defined
  by Definition~\ref{def:taurho-tauin} it holds that
  \[
  y_n\conv{\topin}y \iff \left\{
    \begin{array}{cl}
      y_n \to y & \text{in}\ L^1(\Omega)\\
      \log(y_n)\wstarto \log(y)\ &   \text{in}\ L^\infty(\Omega).
    \end{array}
  \right.
  \]
\end{theorem}
\begin{proof}
  The implication ``$\Rightarrow$'' follows from
  Lemma~\ref{lem:tauIN_characterization} and the fact that
  $D_{KL}(y,y_n)\to 0\Rightarrow \norm[1]{y_n-y}\to 0$
  (see\cite{resmerita2005regbanspaces}). For the converse direction,
  observe that
  \[
  D_{KL}(y,y_n) = \int_\Omega y(\log(y) - \log(y_n))d\mu + \int_\Omega
  y-y_n d\mu
  \]
  and that both terms on the right hand side converge due to the
  assumptions (and the implicit non-negativity assumption).
\end{proof}
Loosely speaking, one can interpret the assumption that $\log(y_n)$
converges weakly$^*$ in $L^\infty(\Omega)$ as a condition that $y_n$
is not allowed to converge to zero on a non-null set
which seems to be a natural condition
in this context.

In view of Theorem~\ref{thm:poeschl_flemming_example} we can
conclude the following: If one aims at Tikhonov regularizing schemes
with Kullback-Leibler divergence $\rho_1(F(x),y)$ and wants to apply
Theorem~\ref{thm:poeschl_flemming_example}, then the appropriate
model for ``data $y^\delta$ converging to noiseless data $y$'' is
given by ``strong convergence in $L^1$ plus weak$^*$ convergence in
$L^\infty$''.

Second, we look at the case of $\rho_2$ as in
Lemma~\ref{lem:tauIN_characterization2}: Here $\rho_2(z,y) =
D_{KL}(y,z)$, i.e., the data is in the first argument of the
Kullback-Leibler divergence as, e.g.,
in~\cite{luke2012linconvapprox}. We conclude directly from
Lemma~\ref{lem:tauIN_characterization2} that a sequence $(y_n)$
converges in $\topin$ to $y$ if and only if
\begin{eqnarray}
  \label{eq:tauin_KL2_1}
  D_{KL}(y_n,y)\to 0\\
  \label{eq:tauin_KL2_2}
  \text{and}\ \int_\Omega
  (\log(y) - \log(z))(y-y_n)d\mu \to 0\ \text{for all}\ z\in \tilde Y.
\end{eqnarray}
In fact the condition in the second line is precisely weak
convergence in $L^1(\Omega)$: Indeed for any $w\in L^\infty(\Omega)$
and $y\in \tilde Y$ we can define $z\in \tilde Y$ via
$z=\exp(w-\log(y))$ and see that
\[
\int_\Omega(\log(y) - \log(z))(y-y_n)d\mu = \int_\Omega w(y-y_n)d\mu 
\]
and hence, the condition is indeed weak convergence in $L^1(\Omega)$.
Since, as already noticed, $D_{KL}(y_y,y)\to 0$ implies that $y_n\to
y$ in $L^1(\Omega)$ strongly, we see that~(\ref{eq:tauin_KL2_1})
implies~(\ref{eq:tauin_KL2_2}) and conclude:
\begin{theorem}
  In the case $\rho_2(z,y) = D_{KL}(y,z)$ and $\topin$ defined by
  Definition~\ref{def:taurho-tauin} it holds that
  \[
  y_n\conv{\topin} y\iff D_{KL}(y_n,y)\to 0.
  \]
\end{theorem}

\begin{remark}
  Note that the convergence in $\topin$ (i.e. $D_{KL}(y_n,y)\to 0$) is
  stronger than strong convergence in $L^1(\Omega)$, even if
  all $y_n$ and $y$ are uniformly bounded away from zero. To see this
  consider the following counterexample: Let $\Omega = [0,1]$,
  $\epsilon>0$ and define $y\equiv\epsilon$. Now define
  \[
  y_n(x) = \left\{
    \begin{array}{ll}
      n & \text{if}\ 0\leq x \leq (n\log(n))^{-1}\\
      \epsilon & \text{if}\  x>(n\log(n))^{-1}
    \end{array}\right.
  \]
  (note that $y,y_n\in \tilde Y$). Then it holds that
  $\norm[1]{y_n-y}\to 0$ but
  \begin{eqnarray*}
    D_{KL}(y_n,y) &=& \int_0^1 y_n(x)\log\Big(\frac{y_n(x)}{y(x)}\Big) - y_n(x) + y(x)dx\\
    &=& \int_0^{(n\log(n))^{-1}} n\log(n/\epsilon) - n + \epsilon dx\\
    &=& \frac{n\log(n/\epsilon) - n + \epsilon}{n\log(n)}\\
    &\to&  1
  \end{eqnarray*}
\end{remark}

As a final remark on the Kullback-Leibler divergence, we note that the
above discussion on topologies deduced from the Kullback-Leibler
divergence when introduced as Bregman distance and considered in
$L^1(\Omega)$, gives another motivation for the use of a positive
``baseline'' when working with Poisson noise,
i.e. the measured data (and hence, also the regularized quantities)
are shifted away from zero by adding a small positive constant as,
e.g., in~\cite{werner2012tikhonovpoisson}.

\section{Conclusion}

% Was wir nicht gemacht haben: Numerik, konkretes Beispiel
% (später?). Bemerkung zu Relevanz? Bemerkung zu weiteren Richtungen
% für notwendige Bedingungen?
We examined variational regularization in a quite general setting and
started a study on necessary conditions for variational schemes to be
regularizing. Although it seems like little can be said about
necessary conditions in general we obtained several results in this
direction. Especially, we tried to clarify the relations between the
different players in the data space, e.g.~the convergence structure,
the topology and the discrepancy functional. Here we started from a
list of conditions which is known to guarantee regularizing properties
and deduced necessary conditions for the topologies and the
discrepancy functional. For Bregman
discrepancies we illustrated that our results imply necessary
conditions for the continuity of the derivative of the functional
which induces the Bregman distance and pointed out structural
difference when the measured data is in the first or second argument
of the Bregman distance, respectively. In the particular case of the
Kullback-Leibler divergence, we also characterized convergence in the
natural topology $\topin$ for both cases.

Although our results are fairly abstract, they are first steps towards
the analysis of necessary conditions which can be used to figure out
essential limitations of variational schemes.  Next steps could be to
analyze the other ingredients of a variational scheme, namely the
solution space $X$, its topology, the regularization functional and of
course, the operator. Other directions for future research are to
consider special classes of discrepancy functionals with additional
structure and to extend the analysis to Morozov and Ivanov regularization.

\section*{References}

\bibliographystyle{plain}
\bibliography{literature}

\begin{thebibliography}{10}

\bibitem{Arkhangelski1990basic.concepts.general.topology}
Alexander~V. Arkhangel{'}ski{\u\i} and Vitaly~V. Fedorchuk.
\newblock Basic concepts and constructions of general topology.
\newblock In L.~S. Pontryagin, editor, {\em General topology. I}, Encyclopaedia
  of Mathematical Sciences. Springer, 1990.

\bibitem{aubert2008multiplicativenoise}
Gilles Aubert and Jean-Fran{\c{c}}ois Aujol.
\newblock A variational approach to removing multiplicative noise.
\newblock {\em SIAM Journal on Applied Mathematics}, 68(4):925--946, 2008.

\bibitem{bakushinksii1984nonregularization}
Anatolii~Borisovich Bakushinski{\u\i}.
\newblock Remarks on choosing a regularization parameter using the
  quasi-optimality and ratio criterion.
\newblock {\em Akademiya Nauk SSSR. Zhurnal Vychislitel$'$ no\u\i\ Matematiki i
  Matematichesko\u\i\ Fiziki}, 24(8):1258--1259, 1984.

\bibitem{beattie2002convergence}
Ronald Beattie and Heinz-Peter Butzmann.
\newblock {\em Convergence structures and applications to functional analysis}.
\newblock Kluwer Academic Publishers, Dordrecht, 2002.

\bibitem{benisrael2003geninverse}
Adi Ben-Israel and Thomas~N.E. Greville.
\newblock {\em Generalized Inverses: Theory and Applications}.
\newblock CMS Books in Mathematics/Ouvrages de Math\'ematiques de la SMC, 15.
  Springer-Verlag, New York, second edition, 2003.

\bibitem{benning2011errorestimates}
Martin Benning and Martin Burger.
\newblock Error estimates for general fidelities.
\newblock {\em Electronic Transactions on Numerical Analysis}, 38:44--68, 2011.

\bibitem{bertero2009poissondeblurring}
M.~Bertero, P.~Boccacci, G.~Desider{\`a}, and G.~Vicidomini.
\newblock Image deblurring with {P}oisson data: from cells to galaxies.
\newblock {\em Inverse Problems}, 25(12):123006, 26, 2009.

\bibitem{borwein1991entropy}
Jonothan~M. Borwein and Adrien~S. Lewis.
\newblock Convergence of best entropy estimates.
\newblock {\em SIAM Journal on Optimization}, 1(2):191--205, 1991.

\bibitem{boyd2004convexoptimization}
Stephen Boyd and Lieven Vandenberghe.
\newblock {\em Convex optimization}.
\newblock Cambridge University Press, Cambridge, 2004.

\bibitem{brune2013fbemtv}
Christoph Brune, Alex Sawatzky, Thomas K{\"o}sters, Frank W{\"u}bbeling, and
  Martin Burger.
\newblock Forward-backward {EM-TV} methods for inverse problems with {P}oisson
  noise.
\newblock to appear as CIME course notes, Lecture Notes in Mathematics, 2013.

\bibitem{daubechies2003iteratethresh}
Ingrid Daubechies, Michel Defrise, and Christine {De Mol}.
\newblock An iterative thresholding algorithm for linear inverse problems with
  a sparsity constraint.
\newblock {\em Communications in Pure and Applied Mathematics},
  57(11):1413--1457, 2004.

\bibitem{diestel1984banachspaces}
Joseph Diestel.
\newblock {\em Sequences and series in {B}anach spaces}, volume~92 of {\em
  Graduate Texts in Mathematics}.
\newblock Springer-Verlag, New York, 1984.

\bibitem{engl1981necessary}
Heinz~W. Engl.
\newblock Necessary and sufficient conditions for convergence of regularization
  methods for solving linear operator equations of the first kind.
\newblock {\em Numerical Functional Analysis and Optimization}, 3(2):201--222,
  1981.

\bibitem{engl1996inverseproblems}
Heinz~W. Engl, Martin Hanke, and Andreas Neubauer.
\newblock {\em Regularization of Inverse Problems}, volume 375 of {\em
  Mathematics and its Applications}.
\newblock Kluwer Academic Publishers Group, Dordrecht, 2000.

\bibitem{flemming2010nonmetric}
Jens Flemming.
\newblock Theory and examples of variational regularization with non-metric
  fitting functionals.
\newblock {\em Journal of Inverse and Ill-Posed Problems}, 18(6):677--699,
  2010.

\bibitem{flemming2011diss}
Jens Flemming.
\newblock {\em Generalized Tikhonov regularization}.
\newblock PhD thesis, Technische Universit{\"a}t Chemnitz, Chemnitz, may 2011.

\bibitem{flemming2011converse}
Jens Flemming, Bernd Hofmann, and Peter Math{\'e}.
\newblock Sharp converse results for the regularization error using distance
  functions.
\newblock {\em Inverse Problems}, 27(2):025006, 18, 2011.

\bibitem{Franklin1965sequencessuffice}
Stanley~P. Franklin.
\newblock Spaces in which sequences suffice.
\newblock {\em Fund. Math.}, 57:107--115, 1965.

\bibitem{frommer1999fastcg}
Andreas Frommer and Peter Maass.
\newblock Fast {CG}-based methods for {T}ikhonov-{P}hillips regularization.
\newblock {\em SIAM Journal on Scientific Computing}, 20(5):1831--1850
  (electronic), 1999.

\bibitem{grasmair2012multiparameter}
Markus Grasmair.
\newblock Multi-parameter {T}ikhonov regularisation in topological spaces.
\newblock To appear in \textit{Journal of Inverse and Ill-posed Problems},
  2012.
\newblock http://arxiv.org/abs/1109.0364.

\bibitem{grasmair2008sparseregularization}
Markus Grasmair, Markus Haltmeier, and Otmar Scherzer.
\newblock Sparse regularization with $\ell^q$ penalty term.
\newblock {\em Inverse Problems}, 24(5):055020 (13pp), 2008.

\bibitem{grasmair2011residual}
Markus Grasmair, Markus Haltmeier, and Otmar Scherzer.
\newblock The residual method for regularizing ill-posed problems.
\newblock {\em Applied Mathematics and Computation}, 218:2693--2710, 2011.

\bibitem{hofmann2007convtikban}
Bernd Hofmann, Barbara Kaltenbacher, Christiane P{\"o}eschl, and Otmar
  Scherzer.
\newblock A convergence rates result for {T}ikhonov regularization in {B}anach
  spaces with non-smooth operators.
\newblock {\em Inverse Problems}, 23(3):987--1010, 2007.

\bibitem{hohage2013irgnpoisson}
Thorsten Hohage and Frank Werner.
\newblock Iteratively regularized {N}ewton-type methods for general data misfit
  functionals and applications to {P}oisson data.
\newblock {\em Numerische Mathematik}, 123(4):745--779, 2013.

\bibitem{ivanov1962wellposed}
Valentin~Konstantinovich Ivanov.
\newblock On linear problems which are not well-posed.
\newblock {\em Doklady Akademii Nauk SSSR}, 145:270--272, 1962.

\bibitem{ivanov1969topological}
Valentin~Konstantinovich Ivanov.
\newblock Ill-posed problems in topological spaces.
\newblock {\em Akademija Nauk SSSR. Sibirskoe Otdelenie. Sibirski\u\i\
  Matemati\v ceski\u\i\ \v Zurnal}, 10:1065--1074, 1969.

\bibitem{ivanov2002linearillposed}
Valentin~Konstantinovich Ivanov, Vladimir~Vasilievich Vasin, and
  Vitali\u{\i}~Pavlovich Tanana.
\newblock {\em Theory of linear ill-posed problems and its applications}.
\newblock Inverse and Ill-posed Problems Series. VSP, Utrecht, 2nd edition,
  2002.

\bibitem{kabanikhin2011inverseproblems}
Sergey~I. Kabanikhin.
\newblock {\em Inverse and Ill-Posed Problems: Theory and Applications}.
\newblock De Gruyter, Berlin/Boston, 2011.

\bibitem{kisynyki1959convergencedutypeL}
Jan Kisy{\'n}ski.
\newblock Convergence du type {${\cal L}$}.
\newblock {\em Colloquium Mathematicum}, 7:205--211, 1959/1960.

\bibitem{lorenz2008reglp}
Dirk~A. Lorenz.
\newblock Convergence rates and source conditions for {T}ikhonov regularization
  with sparsity constraints.
\newblock {\em Journal of Inverse and Ill-Posed Problems}, 16(5):463--478,
  2008.

\bibitem{luke2012linconvapprox}
D.~Russell Luke.
\newblock Local linear convergence of approximate projections onto regularized
  sets.
\newblock {\em Nonlinear Analysis}, 75(3):1531--1546, 2012.

\bibitem{morozov1967discrepancy}
Vladimir~A. Morozov.
\newblock Choice of parameter for the solution of functional equations by the
  regularization method.
\newblock {\em Doklady Akademii Nauk SSSR}, 175(6):1225--1228, 1967.

\bibitem{neubauer1997saturationtikhonov}
Andreas Neubauer.
\newblock On converse and saturation results for {T}ikhonov regularization of
  linear ill-posed problems.
\newblock {\em SIAM Journal on Numeical Analysis}, 34(2):517--527, 1997.

\bibitem{ordman1966nottopological}
Edward~T. Ordman.
\newblock Classroom {N}otes: {C}onvergence {A}lmost {E}verywhere is {N}ot
  {T}opological.
\newblock {\em The American Mathematical Monthly}, 73(2):182--183, 1966.

\bibitem{poeschl2008diss}
Christiane P{\"o}schl.
\newblock {\em Tikhonov Regularization with General Residual Term}.
\newblock PhD thesis, Leopold Franzens Universit{\"a}t Innsbruck, Innsbruck,
  Austria, October 2008.

\bibitem{resmerita2005regbanspaces}
Elena Resmerita.
\newblock Regularization of ill-posed problems in {B}anach spaces: convergence
  rates.
\newblock {\em Inverse Problems}, 21(4):1303--1314, 2005.

\bibitem{resmerita2007kullbackleibler}
Elena Resmerita and Robert~S. Anderssen.
\newblock Joint additive {K}ullback-{L}eibler residual minimization and
  regularization for linear inverse problems.
\newblock {\em Mathematical Methods in the Applied Sciences},
  30(13):1527--1544, 2007.

\bibitem{resmerita2006nonquadreg}
Elena Resmerita and Otmar Scherzer.
\newblock Error estimates for non-quadratic regularization and the relation to
  enhancement.
\newblock {\em Inverse Problems}, 22(3):801--814, 2006.

\bibitem{scherzer2009variationalmethods}
Otmar Scherzer, Markus Grasmair, Harald Grossauer, Markus Haltmeier, and Frank
  Lenzen.
\newblock {\em Variational methods in imaging}, volume 167 of {\em Applied
  Mathematical Sciences}.
\newblock Springer, New York, 2009.

\bibitem{schirotzek2007nonsmoothanalysis}
Winfried Schirotzek.
\newblock {\em Nonsmooth Analysis}.
\newblock Springer, Berlin, 2007.

\bibitem{tanana1975optimality}
Vitali\u{\i}~Pavlovich Tanana.
\newblock The optimality of methods for the solution of unstable nonlinear
  problems.
\newblock {\em Dokl. Akad. Nauk SSSR}, 220:1035--1037, 1975.

\bibitem{tikhonov1963regularization}
Andre\u{\i}~Nikolaevich Tikhonov.
\newblock Solution of incorrecly formulated problems and the regularization
  method.
\newblock {\em Doklady Akademii Nauk SSSR}, 151:501--504, 1963.

\bibitem{vandenberg2007spgl1}
Ewout van~den Berg and Michael~P. Friedlander.
\newblock {SPGL1}: A solver for large-scale sparse reconstruction, June 2007.
\newblock http://www.cs.ubc.ca/labs/scl/spgl1.

\bibitem{vandenberg2008paretofrontiertbasispursuit}
Ewout {van den Berg} and Michael~P. Friedlander.
\newblock Probing the {P}areto frontier for basis pursuit solutions.
\newblock {\em SIAM Journal on Scientific Computing}, 31(2):890--912, 2008.

\bibitem{werner2012tikhonovpoisson}
Frank Werner and Thorsten Hohage.
\newblock Convergence rates in expectation for {T}ikhonov-type regularization
  of inverse problems with {P}oisson data.
\newblock http://arxiv.org/abs/1204.1669, 2012.

\end{thebibliography}
\end{document}